\documentclass[12pt,twoside]{article}
\usepackage[margin=0.9in]{geometry}
\usepackage{fancyhdr}
\usepackage{amsfonts}
\usepackage{amsmath}
\usepackage{amsthm}
\usepackage[normalem]{ulem}
\usepackage{xcolor}
\usepackage{graphicx}
\usepackage{subcaption}

\newcommand{\removed}[1] {\ifmmode{\color{red}\cancel{ #1}}\else{\color{red}\sout{#1}}\fi}

\usepackage{graphicx}
\usepackage{graphics}
\usepackage{epsfig, setspace}
\usepackage{mathrsfs}
\usepackage{amssymb}
\usepackage{amsfonts}
\usepackage{hyperref}
\usepackage{url}

\newcommand{\tnorm}[1]{\ensuremath{\left| \! \left| \! \left|} #1\ensuremath{\right| \! \right| \! \right|}}

\def\sjump#1{[\hskip -.5pt #1 \hskip -0.5pt]}
\def\cT{\mathcal{K}}

\def\cT{\mathcal{T}}

\def\cV{\mathcal{V}}
\def\cF{\mathcal{F}}

\def\cW{\mathcal{W}}

\def\L{\mathcal{L}}
\newcommand{\dx}{\;\textit{dx}}
\newcommand{\ds}{\;\textit{ds}}

\newcommand{\norm}[1]{\left\Vert #1 \right\Vert}

\newtheorem{corollary}{Corollary}[section]

\newtheorem{theorem}{Theorem}[section]
\newtheorem{lemma}{Lemma}[section]

\title{Data assimilation finite element method for the linearized Navier-Stokes equations with higher order polynomial approximation}
\author{Erik Burman\thanks{Department of Mathematics, University College London, London, UK--WC1E 6BT, UK.; {\tt e.burman@ucl.ac.uk}}  \ and Deepika Garg\thanks{Department of Mathematics, University College London, London, UK--WC1E 6BT, UK.; {\tt d.garg@ucl.ac.uk}} \ and Janosch Preuss\thanks{Department of Mathematics, University College London, London, UK--WC1E 6BT, UK.; {\tt j.preuss@ucl.ac.uk}}} 
\date{\today} 

\begin{document}
\maketitle
\begin{abstract}
In this article, we design and analyze an arbitrary-order stabilized finite element method to approximate the unique continuation problem for laminar steady flow described by the linearized incompressible Navier--Stokes equation. We derive quantitative local error estimates for the velocity, which account for noise level and polynomial degree, using the stability of the continuous problem in the form of a conditional stability estimate. Numerical examples illustrate the performances of the method with respect to the polynomial order and perturbations in the data. We observe that the higher order polynomials may be efficient for ill-posed problems, but are also more sensitive for problems with poor stability due to the ill-conditioning of the system. 
\end{abstract}


{\bf Key words:} linearized Navier--Stokes’ equations, data assimilation, stabilized finite element methods,  error estimates


\section{Introduction}
The question of how to assimilate measured data into large-scale computations of flow problems is receiving increasing attention from the computational mathematics community \cite{HMM10, DPV12, BCFGM13, SD18, GN20, ASSS22}. There are several
different situations where such data assimilation problems as can be seen in the above examples. One situation is when the data necessary to make the flow problem well-posed is lacking, for instance, when the data on the boundary of the
the domain is unknown; instead, measurements are available in some subset of the bulk domain or boundary to make up for this shortfall. In such a case, the problem is typically ill-posed, and numerical simulations are significantly more challenging to perform than when handling well-posed flow problems.
Ill-posed problems usually come up in inverse problems and data assimilation.  Traditionally, these ill-posed problems have been solved by regularizing at the continuous level, using e.g. Tikhonov regularization \cite{TA1977} or quasi-reversibility \cite{LL1967}. 
The regularized problem is well-posed and may be discretized using any appropriate numerical technique. Then, the regularization parameter must be tuned to the optimal value for the noise in the data. 
There is considerable literature of research on Tikhonov regularization and inverse problems, and we suggest the reader to \cite{IB2015} and its references for an overview of computational approaches employing this strategy. The quasi-reversibility methods relevant to the current study may be found in \cite{BL2005,BL2006,DH2013, BD14}. 

The goal of the current contribution is to develop a finite element approach directly applied to the ill-posed variational data assimilation form. Regularization is then introduced at the discrete level utilizing stabilized finite element methods that allow for a comprehensive analysis employing conditional stability estimates. The idea is presented in \cite{EB2013} for standard $H^1$-conforming finite element methods.  Ill-posed problems are analyzed in \cite{EB2014}, and in \cite{EB2017}, the technique is extended to nonconforming approximations. In both cases, low-order approximation spaces are considered. The error analysis requires the availability of sharp conditional stability estimates for the continuous problem. The estimates are conditional in the sense that a particular a priori bound must be assumed to hold for the solution, and the continuity provided in this bound is often merely H\"older \cite{john1960continuous}. In the literature, these estimates are referred to as quantitative uniqueness results and employ theoretical methods such as Carleman estimates or three-ball estimates \cite{ALE2009,IV2006}. Error bounds derived using conditional stability estimates can be optimal because they reflect the approximation order of the finite element space and the stability of the ill-posed problem.  In particular, when applied to a well-posed problem, the finite element method recovers optimal convergence.

The ill-posed problem that we consider here is the unique continuation problem. The unique continuation problem for the Stokes equations was initially studied in \cite{fabre1996prolongement}. The analysis of the stability properties of ill-posed problems based on the Navier–Stokes equations is a very active field of research, and we refer to the works \cite{BC2016,AB2010,BIY2016,BEG2013,IM2015,IYM2015,LUW2010} for recent results.

This study aims to determine whether using high-order methods in the primal-dual stabilized Galerkin methods is as helpful in the ill-posed case as in the well-posed situation. Inspired by the approach proposed in \cite{Boulakia:2020:Burman} for the lowest-order finite element discretization of the unique continuation problem subject to the Navier-Stokes equations, here we generalize the method to arbitrary polynomial orders and investigate the benefits of using higher-order polynomials in numerical experiments.

The rest of the paper is organized as follows. In section \ref{sec2}, we introduce the considered inverse problem and some related stability estimates. In section \ref{fem_method}, we describe the proposed stabilized finite element approximation of the data assimilation problem and state the local error estimate. The numerical analysis of the method is carried out in section \ref{sec4}. Finally, section \ref{sec5} presents a series of numerical examples which illustrate the performance of the proposed
method.

\section{The linearized Navier--Stokes problem} \label{sec2}
Let $\Omega$ be an open polygonal (polyhedral) domain in $\mathbb{R}^d$, $d=2,3$.  Let $(U, P)$ be the solution of the stationary incompressible Navier–Stokes equations and consider some perturbation $(u, p)$ of this base flow. If the quadratic term is ignored, the linearized Navier--Stokes equations for $(u, p)$ can be written
\begin{align} 
\mathcal{L}(u,p)  &= {f}; \  \text{\ in} \  \Omega, \label{stoke} \\ \quad \nabla \cdot{u} &= 0   \  \  \text{\ in} \  \Omega,  \label{stoke_1}
\end{align}
where
\begin{align*}
	\mathcal{L}(u,p)=(U \cdot \nabla)u+(u \cdot \nabla)U-\nu \Delta{u} + \nabla p.
\end{align*}
Here, $\nu$ is a diffusion coefficient. We assume that $U$ belongs to $[W^{1,\infty}(\Omega)]^d$ and that $(u, p)$ satisfies the regularity
\[(u,p) \in [H^{2}(\Omega)]^d \times H^{1}(\Omega).\]

For this problem, we assume that measurements on $u$ are available in some subdomain
$\omega_M \subset \Omega$ having a nonempty interior and our purpose is to reconstruct a fluid flow perturbation of $u$ for system (\ref{stoke})--(\ref{stoke_1}) based on the measurements of velocity.

Now, we will present some useful notations. Consider the following spaces:
\[V:=[H^{1}(\Omega)]^d, \ V_0:=[H_0^{1}(\Omega)]^d, \ L_0 :=L^{2}_{0}(\Omega), \ \text{and} \ L :=L^{2}(\Omega) \]
where $L^{2}_{0}(\Omega)=\{ p \in L^{2}(\Omega): \int_{\Omega}p=0\}$. We also define the norms, for $k = 1 \ \text{or} \ d$,
\[\norm{\cdot}_L:=\norm{\cdot}_{[L^{2}(\Omega)]^k}, \ \norm{\cdot}_V:=\norm{\cdot}_{[H^{1}(\Omega)]^k}, \ \norm{\cdot}_{V_{0}^{'}}:=\norm{\cdot}_{[H^{-1}(\Omega)]^k}.\]
Observe that in the definitions, we employ the same notation for $k = 1 \ \text{and} \ k = d.$
 For any subdomain $X \subset \Omega$, we set
\begin{align*}
|v|_{X} :=\left(\int_{X}|v|^{2}\right)^{\frac{1}{2}}, \forall \ v \in [L^{2}(X)]^d.
\end{align*}
Next, define the bilinear forms as: for all $(u, v) \in V \times V$
\begin{align}\label{bi_a}
a(u,v):=\int_{\Omega}((U \cdot \nabla)u+(u \cdot \nabla)U) \cdot v+ \nu \int_{\Omega} \nabla{u} : \nabla{v},
\end{align}	
where $H:G := \sum^{d}_{i,j=1}H_{i,j}G_{i,j}$ and, for all $(p, v) \in L \times V
$
\begin{align}\label{bi_b}
b(p,v):&=\int_{\Omega}p \nabla \cdot v,\\
l(v):&=\int_{\Omega}f  \cdot v.
\end{align}
The weak form of  the inverse problem can be expressed as: $f \in V^{'}_{0}$, $u|_{\omega_M}$
being given, find $(u, p) \in V \times L_0 $ such that
\begin{align}\label{omega}
	u=q\ \text{in} \ \ \omega_M
\end{align}
and
\begin{align}\label{weak_form}
	a(u,v)-b(p,v)+b(r,u)= \langle f,v \rangle_{V^{'}_{0},V_0}, \  \ \forall \ (v,r) \in V_0 \times L.
\end{align}
Here, $q \in [H^1(\omega_M)]^d $ corresponds to the exact fluid velocity on $\omega_M$, i.e. $q$ is a solution to the linearized Navier-Stokes' equations in $\omega_M$ and has an extension $u$ to all of $\Omega$. Below in the finite element method we will assume that we do not have access to $q$, but only some measured velocities $u_M = q + \delta u$. So $u_M$  corresponds to the exact velocity polluted by a small noise $\delta u \in [L^2(\omega_M)]^d$.

Consider the linearized Navier–Stokes problem with a non-zero velocity divergence
\begin{align} 
	\mathcal{L}(u,p)  &= {f}; \  \text{\ in} \  \Omega, \label{stoke_2} \\ \quad \nabla \cdot{u} &= g   \  \  \text{\ in} \  \Omega.  \label{stoke_3}
\end{align}
We assume that if the boundary conditions of system (\ref{stoke_2})--(\ref{stoke_3}) are homogeneous Dirichlet boundary conditions, then it is well-posed. More precisely, we make the following assumption:

 {\bf Assumption A.} For all $f \in V'_0 $ and $g \in L_0$ we assume 
that system (\ref{stoke_2})--(\ref{stoke_3}) admits a unique weak solution $(u, p) \in V_0 \times L_0 $ and that there exists a constant $C_{S} > 0$ depending only on $U, \nu$ and
$\Omega$ such that
\begin{align}
	\norm{u}_V+\norm{p}_L \leq C_S(\norm{f}_{V^{'}_{0}}+\norm{g}_L).
\end{align}
Furthermore, if $\norm{\nabla U}_{[L^{\infty}(\Omega)]^{d \times d}}$ is small enough, then the Lax–Milgram lemma implies that $ \text{Assumption A}$ holds. The assumption of smallness on $\nabla U$ is a sufficient condition, there are reasons to believe that $ \text{Assumption A}$ holds in more general cases.

In the homogeneous case (which corresponds to $f = 0$ in (\ref{stoke})--(\ref{stoke_1}) or to $f = 0$ and $g = 0$ in (\ref{stoke_2})--(\ref{stoke_3})), a
solution $(u, p)$ satisfies a three-balls inequality which only involves the $L^2$ norm of the velocity.
This three-balls inequality result is stated in \cite{LUW2010} (with their notations, $A$ corresponds to $U$ and
$B$ to $\nabla U$). 



\begin{theorem} {\em(Conditional stability for the linearized Navier–Stokes problem).} \label{cor_stab} Let $f \in V'_0$, $\omega_M \subset \Omega$ and $g \in L $ be 
	given. For all $B \subset\subset \Omega, $ there exist $C > 0$ and $0 < \tau < 1$ such that
\begin{align} \label{cond_stab}
	|u|_B \leq C(\norm{f}_{V^{'}_{0}}+\norm{g}_L+\norm{u}_L)^{1-\tau}(\norm{f}_{V^{'}_{0}}+\norm{g}_L+|{u}|_{\omega_M})^{\tau},
\end{align}	
for all $(u, p) \in [H^1(\Omega)]^d \times H^1(\Omega) $ solution of (\ref{stoke_2})--(\ref{stoke_3}).
\end{theorem}
{\bf Proof.}  For the proof we refer the reader to \cite[Appendix A]{Boulakia:2020:Burman}.

Theorem \ref{cor_stab} provides a conditional stability result for ill-posed problems \cite{ALE2009} in the sense that, for this estimate to be helpful, it must be accompanied by an a priori bound on the solution on the global domain (due to the presence of $\norm{u}_L$ on the right-hand side). Specifically, Theorem \ref{cor_stab}  implies that a solution $(u, p)$ in $[H^1(\Omega)]^d \times H^1(\Omega)$ of problem (\ref{omega}) and (\ref{weak_form}), must be unique. For the pressure uniqueness holds up to a constant.
 Moreover, in inequality (\ref{cond_stab}), the exponent $\tau$ depends on the dimension $d$, the size of the
measure domain $\omega_M$ and the distance between the target domain $B$ and the boundary of the computational domain $\Omega$.


Moreover, let $f \in [L^2(\Omega)]^d$ and we introduce the operator $A$ defined on $(V \times L_0) \times (V_0 \times L)$ by
\begin{align}
A((u,p),(v,r)):=a(u,v)-b(p,v)+b(r,u)
\end{align}
where $a$ and $b$ are respectively defined by (\ref{bi_a}) and (\ref{bi_b}). Thus, we look for $(u, p) \in V \times L_0 $ such
that
\begin{align}\label{cont_week}
A((u,p),(v,r))=l(v)	\ \ \forall (v,r) \in V_0 \times L
\end{align}
and (\ref{omega}) holds.
\section{Stabilized finite element approximation} \label{fem_method}
In this section, we first introduce a discretization of problem (\ref{cont_week}) using a standard finite element
method. Then, the discrete inverse problem is reformulated as a constrained minimization problem in the discrete space where the regularization of the cost functional is achieved through
stabilization terms. Finally, the estimation of the error between the exact continuous solution and the discrete solution of our minimization problem is stated in Theorem \ref{th_01} which
corresponds to our main theoretical result.

Let $\{\mathcal{T}_h\}_h$ be a family of affine, simplicial meshes of $\Omega$. For simplicity, the family $\{\mathcal{T}_h\}_h$ is supposed to be quasi-uniform. Mesh faces
are collected in the set $\mathcal{F}_h$ which is split into the set of interior faces, $\mathcal{F}^{int}_h$, and of boundary
faces, $\mathcal{F}^{ext}_h$ . For a smooth enough function $v$ that is possibly double-valued at $F \in \mathcal{F}^{int}_h$ with $F=\partial{T}^{-} \cap \partial{T}^{+}$, we define its jump at $F$ as $\sjump{v}=:v_{T^{-}}-v_{T^{+}}$, and we fix the unit normal vector to $F$, denoted by $\nu_F$, as pointing from $T^-$ to $T^+$. The arbitrariness in the sign of  $\sjump{v}$ is irrelevant in what follows.

 We next define a piecewise polynomial space as
\begin{align*}
\mathbb{P}_{k}(\cT_h):=\left\{v\in {L}^{2}(\Omega): v|_T\in \mathbb{P}_k(T)\quad \forall T\in \cT_h\right\},
\end{align*}
where $\mathbb{P}_k(T)$, $k\ge 0$, is the space of polynomials of degree at most $k$ over the element $T$. Further, define a conforming finite element space as 
\begin{align*}
{P}^{c}_{k}(\cT_h) := \left\{v \in {H}^{1}(\Omega) \ : \ v|_T \in \mathbb{P}_k(T) ~~ \forall ~ T \in \cT_{h}   \right\}.
\end{align*}
Let $ {V}^k_{h}:=[{P}^{c}_{k}(\cT_h)]^{d}$, $ {W}_{h}:=V_0 \cap {V}_{h}^{k_1}$, $Q^0_{h}:={L}_{0}^{2}(\Omega)\cap {P}^{c}_{k_2}(\cT_{h})$ and $Q_{h}:= {P}^{c}_{k_3}(\cT_{h})$. For the analysis below the polynomial degrees of the above spaces may be chosen as $k \ge 1$, $k_1 \ge 1$, $k_2 \in \{ \mbox{max}\{1,k-1\}, \ k\}$ and $k_3 \ge 1$ and the convergence order will be given in terms of $k$. To make the notation more compact we introduce the composite spaces $\cV_h :=
V^k_h \times Q_{h}^0$
 and $\cW_h := W_h \times Q_h.$
We may then write the finite element approximation of (\ref{cont_week}): Find $({u}_h,p_h) \in \mathcal{V}_h $ such that
\begin{align}\label{dis_weak}
    A(({u}_h,p_h),({v}_h,q_h))=l({v}_h),
\end{align}
for all $({v}_h,q_h) \in \mathcal{W}_h $.

 Let us introduce the measurement bilinear form to take into account the measurements on $\omega_M$ given by (\ref{omega}).
\begin{align}
	m(u,u):=|{u}|_{\omega_M}^2=\gamma_M \xi^{-1} \int_{\omega_M} u^2,
\end{align}
where $\xi=\max(\nu, \norm{ U}_{[L^{\infty}(\Omega)]^{d \times d}}h)$ and $\gamma_M > 0$ will correspond to a free parameter representing the relative confidence in the
measurements. The objective is then to minimize the functional
\begin{align}
	\frac{1}{2}m(u_M-u_h, u_M-u_h)
\end{align}
under the constraint that $(u_h, p_h)$ satisfies (\ref{dis_weak}). 

We now introduce the following discrete Lagrangian for $(({u}_h,p_h),({z}_h,y_h)) \in \cV_h \times \cW_h$,
\begin{align}\label{eq4}
\L_h(({u}_h,p_h),({z}_h,y_h)):= \frac{1}{2}m({u}_h-{{u}}_M,{u}_h-{{u}}_M)+A(({u}_h,p_h),({z}_h,y_h))-l({z}_h).
\end{align}
If we differentiate with respect to $({u}_h, p_h)$  and $({z}_h,y_h)$, we get the following optimality system: Find  $ ({u}_h, p_h) \in \cV_h $
and $({z}_h,y_h) \in \cW_h$ such that
\begin{align}
A(({u}_h,p_h),({w}_h,x_h))&=l({w}_h),  \\A(({v}_h,q_h),({z}_h,y_h))+m({u}_h,{v}_h)
&=m({{u}}_{M},{v}_h),
\end{align}
for all $ ({v}_h, q_h) \in \cV_h $ and $({w}_h,x_h) \in \cW_h$. However, the discrete Lagrangian associated to
this problem leads to an optimality system which is ill-posed. 
To regularize it, we introduce stabilization operators that will convexify the problem with respect to the direct variables $u_h,
p_h$ and the adjoint variables $z_h, y_h$. We introduce $S_{{u}}: {V}_h \times {V}_h \rightarrow \mathbb{R}$, $S^{\ast}_{{u}}: {W}_h \times {W}_h \rightarrow \mathbb{R}$, $S_{p}: {Q}^0_h \times {Q}^0_h \rightarrow \mathbb{R}$ and $S^{\ast}_{p}: {Q}_h \times {Q}_h \rightarrow \mathbb{R}$. The choice of stabilization terms will be discussed later.
For compactness, we introduce the primal and dual stabilizers: for all $({u}_h, p_h),({v}_h, q_h) \in \cV_h$
\begin{align}
S_h(({u}_h, p_h),({v}_h, q_h))&=S_g(({u}_h, p_h),({v}_h, q_h))+\tilde{S}_h(({u}_h, p_h),({v}_h, q_h)),\nonumber \\
	S_g(({u}_h, p_h),({v}_h, q_h))&= {\gamma_{GLS} \sum_{T \in \mathcal{T}_h} \int_{T}h^2_T \xi^{-1}_{T} \mathcal{L}(u_h,p_h)\mathcal{L}(v_h,q_h) \dx},
    \label{eq:Sg}
 \\ \tilde{S}_h(({u}_h, p_h),({v}_h, q_h))&= \alpha(h^{2k} \nabla{{u}_h},\nabla{{v}_h})+\gamma_u \sum_{F \in \mathcal{F}^{int}_h} \int_{F}h_F\xi_{F} \sjump{\nabla{{u}_h  \cdot {n}}}\sjump{\nabla{{v}_h \cdot {n}}} \ds \nonumber \\&+\gamma_{div} \int_{\Omega}\xi_{T}(\nabla \cdot {u}_h)(\nabla \cdot {v}_h) \dx,
  \label{eq:Sh}
	\end{align}
 where $\xi_{T}=\max(\nu, \norm{ U}_{[L^{\infty}(\Omega)]^{d \times d}}h_T)$, $\xi_{F}=\max(\nu, \norm{ U}_{[L^{\infty}(\Omega)]^{d \times d}}h_F)$ and $\gamma_{GLS}$, $\alpha$, $\gamma_u$, and $\gamma_{div}$ 
are positive user-defined parameters.
And for all $({z}_h, y_h),({w}_h, x_h) \in \cW_h$
\begin{align}
	S^{\ast}_h(({z}_h, y_h),({w}_h, x_h))&=S^{\ast}_{{u}}({z}_h, {w}_h)+S^{\ast}_p( y_h, x_h),\nonumber \\S^{\ast}_{{u}}({z}_h, {w}_h)&=\gamma^{\ast}_u  \int_{\Omega} {\nabla{{z}_h}}:{\nabla{{w}_h}} \dx, \\ S^{\ast}_p( y_h, x_h) &=\gamma^{\ast}_{p} \int_{\Omega}{y}_h{x}_h \dx, \label{eq:press_stab_l2}
\end{align}
where $\gamma_u^*$ and $\gamma^*_p$ are positive user-defined parameters. Let us make some comments
on these stabilization terms. The stabilization of the direct velocity acts on fluctuations of the
discrete solution through a penalty on the jump of the solution gradient over element faces and
has no equivalent on the continuous level. The form $S_{g}(\cdot, \cdot)$ is a Galerkin least squares stabilization.  Let us mention that there is  some freedom in the choice of dual stabilization, e.g. set $S^{\ast}_p( y_h, x_h) =\gamma^{\ast}_{p} \int_{\Omega} \nabla {y}_h \nabla{x}_h \dx$. We will only detail the analysis for the first choice \eqref{eq:press_stab_l2} below.
We refer the reader to \cite{EB2013,EB2016} for a more general discussion of the possible stabilization operators. 

We may then write the discrete Lagrangian $\L_h:\cV_h \times \cW_h \rightarrow \mathbb{R}$, for all $ ({u}_h, p_h) \in \cV_h $ and $({z}_h,y_h) \in \cW_h$.
\begin{align}
\L_h(({u}_h,p_h),({z}_h,y_h)):= \frac{1}{2}m({u}_h-{{u}}_M,{u}_h-{{u}}_M)+A(({u}_h,p_h),({z}_h,y_h))-l(z_h) \nonumber \\+\frac{1}{2} S_g(({u}_h-u, p_h-p),({v}_h, q_h))+\tilde{S}_h(({u}_h, p_h),({v}_h, q_h))-\frac{1}{2} S^{\ast}_h(({z}_h, y_h),({z}_h, y_h))
\end{align}
If we differentiate with respect to $({u}_h, p_h)$  and $({z}_h,y_h)$, we get the following optimality system: Find  $ ({u}_h, p_h) \in \cV_h $
and $({z}_h,y_h) \in \cW_h$ such that
\begin{align} 
A(({u}_h,p_h),({w}_h,x_h))-S^{\ast}_h(({z}_h, y_h),({w}_h, x_h))&=l({w}_h),  \label{eq_03}\\A(({v}_h,q_h),({z}_h,y_h))+S_h(({u}_h, p_h),({v}_h, q_h))+m({u}_h,{v}_h)
&=m({{u}}_{M},{v}_h) \nonumber\\ &+\gamma_{GLS}\sum_{T \in \cT_h}\int_{T}{f}h_T^2\xi^{-1}_{T}\mathcal{L}(v_h,q_h) \dx,\label{eq_04}
\end{align}
for all $ ({v}_h, q_h) \in \cV_h $ and $({w}_h,x_h) \in \cW_h$.


\section{Stability and Error Analysis} \label{sec4}
 To prove the stability of our formulations, we need the following result.
\begin{lemma} 
  There exists $C_p$ such that for all $v_h \in V_h$ there holds
 \begin{align}\label{poincare_ineq}
     \norm{v_h}_{H^{1}(\Omega)} \leq C_p(\norm{v_h}_{\omega_M}+\norm{\nabla{v_h}}_{L}).
\end{align}
\end{lemma}
\begin{proof} The following Poincar\'e inequality is well known \cite[lemma B.63]{Ern:2004:FEM}. If $f:H^{1}(\Omega) \rightarrow \mathbb{R}$ is a linear functional that is non-zero for constant functions then
\begin{align*}
 \norm{v}_{H^{1}(\Omega)}  \leq C_p (|f(v)|+\norm{\nabla{v}}_{L}), \quad \forall v \in H^{1}(\Omega).
\end{align*}
For instance, we may take
\begin{align*}
    f(v) = \int_{\omega_M} v \dx \leq C |v|_{\omega_M}.
\end{align*}
As an immediate consequence we have the bound (\ref{poincare_ineq}).
\end{proof}
Let us prove that the discrete problem is well-posed. We can write the discrete formulation in a more compact form. Let $({u}_h, p_h)=U_h$, $({z}_h, y_h)=Z_h$, $({v}_h, q_h)=X_h$ and $({w}_h, x_h)=Y_h$.
\begin{align}\label{bilinear}
\mathcal{G}((U_h,Z_h),(X_h,Y_h))=A_h(U_h,Y_h)&-S^{\ast}_h(Z_h,Y_h) +A_h(X_h,Z_h)+S_h(U_h,X_h)+\gamma_M({u}_h,{v}_h)_{\omega_M}.
\end{align}
We define the norm on $([H^{2}(\Omega)]^d+V_h)\times (H^{1}(\Omega)+Q_h)$
\begin{align} \label{eq_711}
    \tnorm {(U_h,Z_h)}^{2} &:= S_h(U_h,U_h)+\gamma_M|{u}_h|_{\omega_M}^{2}+S^{\ast}_h(Z_h,Z_h).
\end{align}
$\tnorm {(U_h,Z_h)}$ defines a norm, since $\gamma_M>0$, $\alpha >0$  and thanks to the Poincar\'{e} inequality (\ref{poincare_ineq}). 
The following result demonstrates the stability of the system (\ref{eq_03})--(\ref{eq_04}). 
\begin{theorem} \label{inf_sup} {The discrete bilinear form (\ref{bilinear}) satisfies the following inf-sup condition for some positive constant $\gamma$, independent of $h$:  }  
\begin{equation*}
 \inf_{(U_h,Z_h) \in \mathcal{V}_h \times\mathcal{W}_h  } \sup_{(X_h,Y_h) \in \mathcal{V}_h \times\mathcal{W}_h  }\frac{\mathcal{G}((U_h,Z_h),(X_h,Y_h))}{\tnorm{(U_h,Z_h)}  \tnorm{(X_h,Y_h) }} \geq \gamma.
\end{equation*}
\end{theorem}
{\bf Proof.} In order to prove the stability result, it is enough to choose some $(X_h,Y_h) \in \mathcal{V}_h \times\mathcal{W}_h$ for any arbitrary $(U_h,Z_h) \in \mathcal{V}_h \times\mathcal{W}_h,$ such that
\begin{equation*} \label{inf_sup_d} 
 \frac{\mathcal{G}((U_h,Z_h),(X_h,Y_h))}{\tnorm {(X_h,Y_h)}} \geq \gamma \tnorm{(U_h,Z_h)} > 0.
\end{equation*}
First, consider the bilinear form in (\ref{bilinear}) with  $(X_h, Y_h) = (U_h, -Z_h)$: 
\begin{align*}\label{eq11}
    \mathcal{G}((U_h,Z_h),(U_h,-Z_h))&=S^{\ast}_h(Z_h,Z_h)+S_h(U_h,U_h)+\gamma_M({u}_h,{u}_h)_{\omega_M} \nonumber \\ &=S^{\ast}_h(Z_h,Z_h)+S_h(U_h,U_h)+\gamma_M|{u}_h|^{2}_{\omega_M}.
\end{align*}
\begin{equation} \label{eq8}
\mathcal{G}((U_h,Z_h),(U_h,-Z_h)) \geq \tnorm{(U_h,Z_h)}^{2}.
\end{equation}
and 
\begin{align} \label{eq9}
    \tnorm{(U_h,-Z_h)} \leq \tnorm{(U_h,Z_h)}.
\end{align}
Finally, by dividing (\ref{eq8}) by (\ref{eq9}), we get the result.

According to the Babu\v{s}ka--Ne\v{c}as--Brezzi theorem (see \cite{Ern:2004:FEM}), the square linear system defined by (\ref{eq_03})--(\ref{eq_04}) admits a unique solution for all $h > 0$.

\subsection{Error Analysis}
Now recall the following technical results of finite element analysis.
\begin{lemma}{\rm Trace inequality} \cite{Ern:2012:DGBook}: {Suppose $F$ denotes an edge of $T \in \cT_{h}$}. For  $v_h\in \mathbb{P}_{k}(\cT_{h})$, there holds
	\begin{align} 
	\|v_h\|_{{L}^2(F)} &\leq C h_T^{-1/2} \|v_h\|_{{L}^2(T)} . \label{trace_ineq1}
	\end{align}
\end{lemma}
\begin{lemma} {\rm Inverse inequality} \cite{Ern:2012:DGBook}: Let $v \in \mathbb{P}_{k}(\cT_{h})$, for all $k \geq 0$. Then,
\begin{align}
\norm{\nabla{v}}_{{L}^2(T)} \leq C h^{-1}_{T} \norm{v}_{{L}^2(T)}.  \label{inverse_ineq1}
\end{align}
\end{lemma}

\begin{lemma} Let  $I_h:{L}^2(\Omega)\rightarrow  {P}^{c}_{k}(\cT_h)$ be the Cl$\acute{e}$ment interpolation. The following approximation estimates hold for the interpolation operator $I_h$, see \cite{Ern:2004:FEM},
	\begin{align}\label{eq_Istab}
	   \norm{ I_{h}v}_L \leq C \norm{v}_L, \forall v \in L\quad \norm{ \nabla I_{h}v}_L \leq C \norm{\nabla v}_L, \forall v \in H^1(\Omega),
    \end{align}
	\begin{align}\label{eq_30}
	\norm{ ({v-I_{h}v})}_L+h\norm{\nabla(v-I_hv)}_L &\leq Ch^{t}\norm{v}_{{H}^{t}(\Omega)},\ {\rm for~ all } \ v\in {H}^{t}(\Omega), \  1 \leq t \leq k+1, \\
 \left(\sum_{T \in \mathcal{T}_h}\norm{\Delta({v}-{I}_h{v})}^2_{{L}^2(T)} \right)^{1/2}&\leq Ch^{ t-2}\norm{{v}}_{{H}^t(\Omega)},\ {\rm for~ all } \ v\in {H}^{t}(\Omega), \ 2 \leq t \leq k+1,\label{eq_30_1}
	\end{align}
		\begin{align}\label{intglobal_trace}
	\left(\sum_{F \in \mathcal{F}^{int}_h}\norm{{v}-{I}_h{v}}^2_{{L}^2(F)} \right)^{1/2}&\leq Ch^{ t-1/2}\norm{{v}}_{{H}^t(\Omega)},\ {\rm for~ all } \ v\in {H}^{t}(\Omega), \ t \leq k+1,\\
	\left(\sum_{F \in \mathcal{F}^{int}_h}\norm{\nabla({v}-{I}_h{v})}^2_{{L}^2(F)} \right)^{1/2}&\leq Ch^{ t-3/2}\norm{{v}}_{{H}^t(\Omega)},\ {\rm for~ all } \ 1 \leq v\in {H}^{t}(\Omega), \  2 \leq t \leq k+1.	\label{intglobal_trace_1}
	\end{align}
The same bound holds for interpolation of vector-valued functions,  ${I}_{h}:[{L}^2(\Omega)]^{d}\rightarrow  V^k_h$ and for interpolation on $\cW_h$ where homogeneous boundary conditions are imposed.
\end{lemma}
Using the above bounds to the componentwise extension of $I_h$ to vectorial functions, we deduce the following approximation bound. 
\begin{corollary}  \label{cor_L} It holds for $(u,p) \in [H^{k+1}(\Omega)]^d \times  H^{k}(\Omega)$,
    \begin{align*}
       \left( \sum_{T \in \mathcal{T}_h}\norm{ \mathcal{L}(I_hu-u,I_hp-p)}_{L^2(T)}^2\right)^{\frac12} \leq C h^{k-1} \left(\norm{ {u}}_{[H^{k+1}(\Omega)]^d} + \norm{p}_{H^{k}(\Omega)} \right ).
    \end{align*}
\end{corollary}

In this section, we will present and prove several technical results. First observe that the formulation (\ref{eq_03})--(\ref{eq_04}) is weakly consistent in the sense that we have a
modified Galerkin orthogonality relation with respect to the scalar product associated to $A$:

\begin{lemma}\label{lemma_cons} {\em(Consistency).}
Let $({u}, p)$ satisfy (\ref{stoke}) and $({u}_h, p_h)$ be a solution of (\ref{eq_03})--(\ref{eq_04}). Then
there holds
\begin{align}
  A(({u}-{u}_h, p-p_h), ({w}_h,x_h) ) =-S^{\ast}_h(({z}_h,y_h),({w}_h,x_h)), \quad \forall ({w}_h,x_h) \in \cW_h.
\end{align}
\end{lemma}
{\bf Proof.} The result follows by taking the difference between (\ref{cont_week}) and (\ref{eq_03}).
\begin{lemma} \label{lemma_approx}   Let $({u}, p) \in [{{H}^{k+1}}(\Omega)]^{d}\times  {L}^{2}_{0} \bigcap {H}^{k}(\Omega)$. Then,
\begin{align} \label{eq_31_stoke}
\tnorm {({{u}}-{I}_h {{u}}, p-I_h{p})} \leq C h^{k}\left(\norm{ {u}}_{[H^{k+1}(\Omega)]^d} + \norm{ p}_{H^{k}(\Omega)} \right ).
\end{align}
\end{lemma}
{\bf Proof.} The approximation bounds can be deduced using the component-wise extension of $I_h$ to vector functions.

\begin{lemma} {\em(Continuity).} \label{lemma_conti}   Let $({u}, p) \in [{{H}^{k+1}}(\Omega)]^{d}\times  {L}^{2}_{0} \bigcap {H}^{k}(\Omega)$. Then,
\begin{align} \label{lemma_conti_1}
   A(({u}-I_h {u}, p-I_h p),({z}_h,y_h)) &\leq C h^{k} \left(\norm{ {u}}_{[H^{k+1}(\Omega)]^d} + \norm{p}_{H^{k}(\Omega)} \right ) S^{\ast}_h(({z}_h,y_h),({z}_h,y_h))^{\frac{1}{2}}, 
\end{align}
for all $({z}_h,y_h) \in \cW_h$.
\end{lemma}
{\bf Proof.} Let us derive the estimate (\ref{lemma_conti_1}). Using the definition of $A(\cdot, \cdot):$
\begin{align} \label{eq_09}
  A((I_h {u}-{u}, I_h p-p),({z}_h,y_h))  =a(I_h {u}-{u},{z}_h) -b(I_h p-p,{z}_h)+b(y_h,I_h {u}-{u}).
\end{align}
Consider the first term on the right hand side of (\ref{eq_09}). Using the Cauchy--Schwarz inequality and Poincar\'{e} inequality,
\begin{align*}
  a(I_h {u}-{u},{z}_h)&\leq  C\norm{U}_{[W^{1,\infty}]^d}(\norm{{u}-{I}_h{{u}} }_L +\norm{\nabla({u}-{I}_h{{u}} )}_L) (\norm{ {z}_h}_L+\norm{\nabla {z}_h}_L) \nonumber \\ &\leq C \norm{U}_{[W^{1,\infty}]^d}h^{k} \norm{{u}}_{[H^{k+1}(\Omega)]^d}S^{\ast}_h(({z}_h,0),({z}_h,0))^{\frac{1}{2}}.
\end{align*}
The second term of (\ref{eq_09}) can be handled as:
\begin{align*}
    b(I_h p-p,{z}_h)\leq \norm{p-{I}_h p }_L  \norm{\nabla \cdot {z}_h}_L \nonumber \\ \leq C h^{k} \norm{{p}}_{H^{k}(\Omega)}S^{\ast}_h(({z}_h,0),({z}_h,0))^{\frac{1}{2}}.
\end{align*}
The last term of (\ref{eq_09}) can be handled as:
\begin{align*}
  b(y_h,I_h {u}-{u})  \leq  \norm{\nabla \cdot ({u}-{I}_h{{u}} )}_L  \norm{y_h}_L \nonumber \\ \leq C h^{k} \norm{{u}}_{k+1}S^{\ast}_h((0,y_h),(0,y_h))^{\frac{1}{2}}.
\end{align*}
 Finally, the result follows by combining all
the above estimates.
\begin{lemma} \label{lemma_01} We assume that the solution $({u}, p) \in [{{H}^{k+1}}(\Omega)]^{d}\times  {L}^{2}_{0} \cap {H}^{k}(\Omega)$ and we consider $({u}_{h}, {p}_{h}) \in \mathcal{V}_h$ and $({z}_{h}, {q}_{h}) \in \mathcal{W}_h$ the discrete solution of (\ref{eq_03})--(\ref{eq_04}).  Then there holds,
\begin{align} \label{eq_02}
\tnorm {({{u}}- {{u}_h}, p-{p}_h),({z}_{h}, {q}_{h})} \leq C \Bigl( h^{k}(\norm{ {u}}_{[H^{k+1}(\Omega)]^d} + \norm{p}_{H^{k}(\Omega)} ) +\gamma^{\frac{1}{2}}_m |\delta {u}|_{\omega_M}\Bigr ).
\end{align}
\end{lemma}
{\bf Proof.} We introduce the discrete errors $\zeta_h= I_h {u}-{u}_h$, $\eta_h= I_h p-p_h$. By this way, 
\begin{align}\label{eq_05}
    \tnorm {({{u}}- {{u}_h}, p-{p}_h),({z}_{h}, {q}_{h})} \leq\tnorm{({{u}}-{I}_h {{u}}, p-I_h{p}),(0,0)}+\tnorm{(\zeta_{h}, \eta_{h}),({z}_{h}, {q}_{h})}.
\end{align}
The first term of (\ref{eq_05}) can be handled by using the Lemma \ref{lemma_approx}. 
Consider the second term of (\ref{eq_05})
\begin{align*}
   \tnorm{(\zeta_{h}, \eta_{h}),({z}_{h}, {q}_{h})}^{2} =S^{\ast}_h(({z}_{h}, {q}_{h}),({z}_{h}, {q}_{h}))+S_h((\zeta_{h}, \eta_{h}),(\zeta_{h}, \eta_{h})) +\gamma_M|\zeta_{h}|_{\omega_M}^{2}. 
\end{align*}
To estimate the right-hand side, we notice that, using the second equation of (\ref{eq_04}) with $(v_h, q_h) =
(\zeta_{h}, \eta_{h})$
\begin{align}\label{eq_06}
    S_h((\zeta_{h}, &\eta_{h}),(\zeta_{h}, \eta_{h})) +\gamma_M|\zeta_{h}|_{\omega_M}^{2}- A((\zeta_{h}, \eta_{h}), ({z}_h,q_h))\nonumber \\&=S_h(I_hu,I_hp),(\zeta_{h}, \eta_{h}))  +m(I_hu-u, \zeta_h)-m(\delta u, \zeta_h)-\gamma_{GLS}\sum_{T \in \cT_h}\int_{T}{f}h_T^2\xi^{-1}_{T}\mathcal{L}(\zeta_h,\eta_h) \dx.
\end{align}
Using Lemma \ref{lemma_cons}, we obtained 
\begin{align}\label{eq_07}
    A(({{u}}-{I}_h {{u}}, p-I_h{p}),(z_h,q_h))+A((\zeta_{h}, \eta_{h}),(z_h,q_h))=-S^{\ast}_h((z_h,q_h),(z_h,q_h)).
\end{align}
Adding (\ref{eq_06}) and (\ref{eq_07}),
\begin{align}
& S^{\ast}_h(({z}_{h}, {q}_{h}),({z}_{h}, {q}_{h}))+S_h((\zeta_{h}, \eta_{h}),(\zeta_{h}, \eta_{h})) +\gamma_M|\zeta_{h}|_{\omega_M}^{2} \nonumber \\  &=\underbrace{A((I_h {u}-{u}, I_h p-p),(z_h,y_h))} \nonumber \\&+\underbrace{S_h(I_hu,I_hp),(\zeta_{h}, \eta_{h}))-\gamma_{GLS}\sum_{T \in \cT_{h} } \int_{T}h_T^2\xi^{-1}_{T}\mathcal{L}(u,p)\mathcal{L}(\zeta_h,\eta_h))} \dx\nonumber\\ & +\underbrace{m(I_hu-u, \zeta_h)-m(\delta u, \zeta_h)} \nonumber\\ &=(1) +(2)+(3). \label{eq_08}
\end{align}
We bound the terms  (1)--(3) term by term. The first term is handled by using  Lemma \ref{lemma_conti}
\begin{align}
   A((I_h {u}-{u}, I_h p-p),({z}_h,q_h)) \leq C h^{k} \left(\norm{ {u}}_{[H^{k+1}(\Omega)]^d} + \norm{p}_{H^{k}(\Omega)} \right )S^{\ast}_h(({z}_h,0),({z}_h,0))^{\frac{1}{2}}.
\end{align}
Consider the second term on the right hand side of (\ref{eq_08})
\begin{align}\label{eq_10}
S_h(I_h {u},I_hp)&,(\zeta_{h}, \eta_{h}))-\gamma_{GLS}\sum_{T \in \cT_{h} } \int_{T}h_T^2\xi^{-1}_{T}\mathcal{L}(u,p)\mathcal{L}(\zeta_h,\eta_h) \dx \nonumber\\ &=(h^{2k} \nabla{I_h {u}_h},\nabla \zeta_{h})+{\gamma_{GLS} \sum_{T \in \cT_{h}} \int_{T}h_T^2\xi^{-1}_{T} \mathcal{L}(I_hu,I_hp)\mathcal{L}(\zeta_h,\eta_h) \dx} \nonumber \\&-\gamma_{GLS}\sum_{T \in \cT_{h} } \int_{T}h_T^2\xi^{-1}_{T}\mathcal{L}(u,p)\mathcal{L}(\zeta_h,\eta_h) \dx+\gamma_u \sum_{F \in \mathcal{F}^{int}_h} \int_{F}h_F \xi_{F}\sjump{\nabla{I_h {u}_h  \cdot {n}}}\sjump{\nabla{\zeta_{h} \cdot {n}}} \ds \nonumber\\&+\gamma_{div} \int_{\Omega}\xi_{T}(\nabla \cdot I_h {u}_h)(\nabla \cdot \zeta_{h}) \dx.
\end{align}
We now estimate the terms on the right hand side of (\ref{eq_10}). Using the $H^{1}$-stability of $I_h$, the first term of (\ref{eq_10}) can be handled as:
\begin{align*}
 (h^{2k} \nabla{I_h {u}},\nabla \zeta_{h})   \leq C h^{k} \norm{{u}}_{[H^{1}(\Omega)]^d} S_h((\zeta_{h}, \eta_{h}),(\zeta_{h}, \eta_{h}))^{\frac{1}{2}}.
\end{align*}
Consider the next two terms of (\ref{eq_10}). Using the Cauchy-- Schwarz inequality and Corollary \ref{cor_L} we obtain
\begin{align*}
   \gamma_{GLS} \sum_{T \in \cT_{h}} &\int_{T}h_T^2 \xi^{-1}_{T}\mathcal{L}(I_hu,I_hp)\mathcal{L}(\zeta_h,\eta_h) \dx-\gamma_{GLS}\sum_{T \in \tau_{h} } \int_{T}h_T^2\xi^{-1}_{T}\mathcal{L}(u,p)\mathcal{L}(\zeta_h,\eta_h) \dx \nonumber\\ =&  {\gamma_{GLS} \sum_{T \in \cT_{h}} \int_{T}h_T^2\xi^{-1}_{T}\mathcal{L}(I_hu-u,I_hp-p)\mathcal{L}(\zeta_h,\eta_h) \dx} \nonumber \\ &\leq \left(\gamma_{GLS}\sum_{T \in \cT_{h}} \int_{T}h_T^2 \xi^{-1}_{T}\mathcal{L}(I_hu-u,I_hp-p)^2 \dx \right)^{\frac{1}{2}}\left(\gamma_{GLS}\sum_{T \in \cT_{h}} \int_{T}h_T^2 \xi^{-1}_{T}\mathcal{L}(\zeta_h,\eta_h)^2 \dx \right)^{\frac{1}{2}} \nonumber \\ &\leq C h^{k} \left(\norm{ {u}}_{[H^{k+1}(\Omega)]^d} + \norm{p}_{H^{k}(\Omega)} \right ) S_h((\zeta_{h}, \eta_{h}),(\zeta_{h}, \eta_{h}))^{\frac{1}{2}}.
\end{align*}
The next term of (\ref{eq_10}) can be handled by using the Cauchy--Schwarz inequality and the estimate (\ref{intglobal_trace_1}),
\begin{align*}
  \gamma_{u}\sum_{F \in \cF^{int}_h} \int_{F}h_F\xi_{F} &\sjump{\nabla{I_h {u}_h  \cdot {n}}}\sjump{\nabla{\zeta_{h} \cdot {n}}} \ds \\ &\leq \left(\gamma_{u}\sum_{F \in \mathcal{F}^{int}_h} \int_{F}h_F \xi_{F}\sjump{\nabla{I_h {u}_h  \cdot {n}}}^2 \ds \right)^{\frac{1}{2}}\left(\gamma_{u}\sum_{F \in \mathcal{F}^{int}_h} \int_{F}h_F \xi_{F}\sjump{\nabla{\zeta_{h} \cdot {n}}}^2 \ds \right)^{\frac{1}{2}} \nonumber \\ &\leq \left(\gamma_{u}\sum_{F \in \mathcal{F}^{int}_h} \int_{F}h_F\xi_{F} \sjump{\nabla{(I_h {u}- {u}) \cdot {n}}}^2 \ds \right)^{\frac{1}{2}}S_h((\zeta_{h}, \eta_{h}),(\zeta_{h}, \eta_{h}))^{\frac{1}{2}}  \nonumber \\ &\leq C h^{k} \norm{ {u}}_{[H^{k+1}(\Omega)]^d}  S_h((\zeta_{h}, \eta_{h}),(\zeta_{h}, \eta_{h}))^{\frac{1}{2}}.
\end{align*}
Put together (\ref{eq_10}) leads to
\begin{align*}
  S_h(I_h {u},I_hp)&,(\zeta_{h}, \eta_{h}))-\gamma_{GLS}\sum_{T \in \cT_{h} } \int_{T}h_T^2\xi^{-1}_{T}\mathcal{L}(u,p)\mathcal{L}(\zeta_h,\eta_h) \dx \\ &\leq C h^{k} \left(\norm{ {u}}_{[H^{k+1}(\Omega)]^d} + \norm{p}_{H^{k}(\Omega)} \right )S_h((\zeta_{h}, \eta_{h}),(\zeta_{h}, \eta_{h}))^{\frac{1}{2}}.
\end{align*}
The last term can be handled as:
\begin{align*}
| m(I_h {u}-{u}, \zeta_h)-m(\delta {u}, \zeta_h)  | &\leq C (|I_h {u}-{u}|_{\omega_M} +|\delta {u}|_{\omega_M}) \gamma_{m}|\zeta_h|_{\omega_M} \nonumber \\ &\leq C (h^{k+1} \norm{ {u}}_{[H^{k+1}(\Omega)]^d} + \gamma^{\frac{1}{2}}_m |\delta {u}|_{\omega_M}) \gamma^{\frac{1}{2}}_{m} |\zeta_h|_{\omega_M}.
\end{align*}
Put together (\ref{eq_08}) leads to
\begin{align*}
  \tnorm{(\zeta_{h}, \eta_{h}),({z}_{h}, {q}_{h})}^{2} &\leq C\Bigl( h^{k} \Bigl(\norm{ {u}}_{[H^{k+1}(\Omega)]^d} + \norm{p}_{H^{k}(\Omega)} \Bigr )+ \gamma^{\frac{1}{2}}_m |\delta {u}|_{\omega_M} \Bigr)\tnorm{(\zeta_{h}, \eta_{h}),({z}_{h}, {q}_{h})} \nonumber \\ \Rightarrow \tnorm{(\zeta_{h},  \eta_{h}),({z}_{h}, {q}_{h})}&   \leq C( h^{k} \Bigl(\norm{ {u}}_{[H^{k+1}(\Omega)]^d} + \norm{p}_{H^{k}(\Omega)} \Bigr )+ \gamma^{\frac{1}{2}}_m |\delta {u}|_{\omega_M}).
\end{align*}
{The combination of the above estimates concludes the claim.}
\begin{corollary}\label{cor_cong}
Under the same assumptions as for Lemma \ref{lemma_01}, there holds
\begin{align} \label{eq_11}
\norm {{u}-{u}_h}_{V} \leq C \Bigl(\norm{ {u}}_{[H^{k+1}(\Omega)]^d} + \norm{p}_{H^{k}(\Omega)}  +h^{-k}\gamma^{\frac{1}{2}}_m |\delta {u}|_{\omega_M}\Bigr ), 
\end{align}
and 
\begin{align} \label{equa011}
\norm{{u}_h}_{V}\leq C\Bigl(\norm{ {u}}_{[H^{k+1}(\Omega)]^d} + \norm{p}_{H^{k}(\Omega)}  +\gamma^{\frac{1}{2}}_m h^{-k} |\delta {u}|_{\omega_M}\Bigr ).
\end{align}
  \end{corollary}
{\bf Proof.} Using Lemma \ref{lemma_01}, we see that
\begin{align*}
    \norm {{u}-{u}_h}_{V}= h^{-k} \norm {h^{k}({u}-{u}_h)}_{V} \leq & Ch^{-k}(S_h({u}-{u}_h,{u}-{u}_h)+|{u}-{u}_h|_{\omega_M}^2) \\ \leq & Ch^{-k}\Bigl(h^{k} (\norm{ {u}}_{[H^{k+1}(\Omega)]^d} + \norm{p}_{H^{k}(\Omega)}  ) +\gamma^{\frac{1}{2}}_m |\delta{u}|_{\omega_M}\Bigr) \\\leq &C \Bigl(\norm{{u}}_{[H^{k+1}(\Omega)]^d} + \norm{p}_{H^{k}(\Omega)}  +h^{-k}\gamma^{\frac{1}{2}}_m |\delta {u}|_{\omega_M}\Bigr ).
\end{align*}
The estimate (\ref{equa011}) is immediate by using the triangle inequality and the estimate (\ref{eq_11}).

The following theorem is the main theoretical result of the paper and states an error estimate for this method.
\begin{theorem} \label{th_01}
Let ${f} \in [L^{2}(\Omega)]^d$ and  ${u}_{M}={u}|_{{\omega_M}}+\delta {u}$ be given. We assume that $({u}, p) \in [{{H}^{k+1}}(\Omega)]^{2}\times  {L}^{2}_{0} \cap H^{k}(\Omega)$ is the solution of (\ref{cont_week}), and consider   $({u}_{h}, {p}_{h}) \in \mathcal{V}_h$ and $({z}_{h}, {y}_{h}) \in \mathcal{W}_h$ the discrete solution of  (\ref{eq_03})--(\ref{eq_04}). Then for all $B \subset \subset \Omega$ there exists $\tau \in (0,1)$ such that 
  \begin{align} \label{eq_12}
 |{u}-{u}_h|_{B}  \leq C h^{k\tau}\Bigl(\norm{ {u}}_{[H^{k+1}(\Omega)]^d} + \norm{ p}_{H^{k}(\Omega)} +h^{-k} |\delta {u}|_{\omega_M}\Bigr).
  \end{align}
  \end{theorem}
{\bf Proof.} Let us first consider the weak formulation of the problem satisfied by $(\zeta, \eta)=({u}-{u}_h, p-p_h)$.
\begin{align*}
A((\zeta, \eta),({w},r))= ({f},{w})_{L}- A(({u}_h, p_h),({w},r)).
\end{align*}
We introduce ${u}_h$ and $p_h$ being fixed. The linear forms $r_f$ and $r_g$ on $V_0$ and $L$ respectively defined by: For all ${w} \in V_0$ and $r \in L$
\begin{align}\label{eq_15_1}
    \langle r_f,{w} \rangle_{V'_{0},V} +(r_g,r)_L:=({f},{w})_{L}- A(({u}_h, p_h),({w},r)).
\end{align}
It follows that $(\zeta, \eta)$ is the solution of (\ref{stoke_2})--(\ref{stoke_3}) with $f$ and $g$ in the right hand sides replaced respectively by $r_f$ and $r_g$. Applying now corollary \ref{cor_stab}, we directly get
\begin{align}\label{eq_131}
    |\zeta|_{B}\leq C (\norm{r_f}_{V'_0}+\norm{r_g}_{L} + \norm{\zeta}_{L})^{1-\tau} (\norm{r_f}_{V'_0}+\norm{r_g}_{L} + |\zeta|_{\omega_M})^{\tau}.
\end{align}
Using (\ref{eq_03}), we can write the residuals: for all $({w}_h,q_h) \in \cW_h$
\begin{align}\label{eq_13}
    <r_f,{w}>_{V'_{0},V} +(r_g,r)_L:=({f},{w}-{w}_h)_{L}- A(({u}_h, p_h),({w}-{w}_h,r-q_h))-S_h^{\ast}(({z}_{h}, {y}_{h}),({w}_h,q_h)).
\end{align} 

We take ${w}_{h}=I_h{w}$ and $q_h=I_h r$ in (\ref{eq_13}). Now, let us estimate the terms on the right hand side of (\ref{eq_13}). Consider the first two terms of (\ref{eq_13}) 
\begin{align}\label{eq_14}
 ({f},{w}-{w}_h)_{L}- &A(({u}_h, p_h),({w}-{w}_h,r-q_h))  \nonumber \\&=({f},{w}-{w}_h)_{L}-(a({u}_h,{w}-{w}_h)  -b(p_h,{w}-{w}_h))+b(r-q_h,{u}_h)).
\end{align}
Applying an integration by parts to the first two terms of (\ref{eq_14}) and using Lemma \ref{lemma_01},
\begin{align}
({f},{w}-{w}_h)_{L}-&(a({u}_h,{w}-{w}_h) -b(p_h,{w}-{w}_h) ) \nonumber \\&=|\sum_{T \in \cT_h}\int_{T}\mathcal{L}(u,p)( {w}-{w}_h) \dx-\sum_{T \in \cT_h}\int_{T}\mathcal{L}(u_h,p_h)( {w}-{w}_h) \dx | \nonumber \\ &+ \sum_{F \in \mathcal{F}^{int}_h} \int_{F}|\sjump{\nabla ({u}-{u}_h) \cdot {n}}||({w}-{w}_h)| \ds \nonumber \\ &=|\sum_{T \in \cT_h}\int_{T}\mathcal{L}(u-u_h,p-p_h)( {w}-{w}_h) \dx| +\sum_{F \in \mathcal{F}^{int}_h} \int_{F}|\sjump{\nabla ({u}-{u}_h) \cdot {n}}||({w}-{w}_h)| \ds \nonumber \\& \leq C\left(\sum_{T \in \cT_h}h_T^2\xi^{-1}_{T}\norm{\mathcal{L}(u-u_h,p-p_h)}_{L^{2}(T)}^2\right)^{\frac{1}{2}} h^{-1} \norm{{w}-{w}_h}_L\nonumber \\& +C\left(\sum_{F \in \mathcal{F}^{int}_h} \int_{F}h_F \xi_{F}\sjump{\nabla ({u}-{u}_h) \cdot {n}}^2 \ds\right)^{\frac{1}{2}} \left(\sum_{F \in \mathcal{F}^{int}_h} \int_{F}h^{-1}_F\xi^{-1}_{F}({w}-{w}_h)^2 \ds\right)^{\frac{1}{2}} \nonumber \\& \leq  C\tnorm{({u} -{u}_h,p- {p}_h),(z_h,q_h)}h^{-1} \norm{{w}-{w}_h}_L \nonumber \\&+ C\left(\sum_{F \in \mathcal{F}^{int}_h} \int_{F}h_F\xi_{F}\sjump{\nabla ({u}-{u}_h) \cdot {n}}^2 \ds\right)^{\frac{1}{2}} \left(\sum_{F \in \mathcal{F}^{int}_h} \int_{F}h^{-1}_F \xi^{-1}_{F}({w}-{w}_h)^2 \ds\right)^{\frac{1}{2}} \nonumber \\ & \leq C(h^{k}\norm{{u}}_{[H^{k+1}(\Omega)]^d}+h^{k} \norm{p}_{H^{k}(\Omega)}) \norm{{w}}_{[H^{1}(\Omega)]^d}+h^{k}\norm{{u}}_{[H^{k+1}(\Omega)]^d}\norm{{w}}_{[H^{1}(\Omega)]^d} \nonumber \\ & \leq Ch^{k}(\norm{{u}}_{[H^{k+1}(\Omega)]^d}+ \norm{p}_{H^{k}(\Omega)}) \norm{{w}}_{[H^{1}(\Omega)]^d}.
\end{align}
The last term is handled using the Cauchy--Schwarz inequality and  Lemma \ref{lemma_01},
\begin{align}
 b(r-q_h,{u}_h)   &\leq \norm{r-q_h}_L \norm{\nabla \cdot {u}_h}_L \nonumber \\ &\leq  \norm{r}_{L} \norm{\nabla \cdot {u}_h}_L \nonumber \\ &\leq C(h^{k}(\norm{{u}}_{[H^{k+1}(\Omega)]^d}+ \norm{p}_{H^{k}(\Omega)})+\gamma^{1/2}_m |\delta {u}|_{\omega_M})\norm{r}_{L}.
\end{align}
Applying the above bounds in  (\ref{eq_14}) leads to
\begin{align}
  ({f},{w}-{w}_h)_{L}- &A(({u}_h, p_h),({w}-{w}_h,r-q_h)) \\ \nonumber &\leq C(h^{k}(\norm{{u}}_{[H^{k+1}(\Omega)]^d}+ \norm{p}_{H^{k}(\Omega)})+\gamma^{1/2}_m |\delta {u}|_{\omega_M}) (\norm{{w}}_{[H^{1}(\Omega)]^d}+\norm{r}_{L}).
\end{align}

The last term of (\ref{eq_13}) is handled using Lemma \ref{lemma_01}
\begin{align}
    S_h^{\ast}(({z}_{h}, {y}_{h}),({w}_h,q_h)) &\leq  S_h^{\ast}(({z}_{h}, {y}_{h}),({z}_h,y_h))^{\frac{1}{2}}S_h^{\ast}(({w}_{h}, {q}_{h}),({w}_h,q_h))^{\frac{1}{2}}  \nonumber \\ &\leq C\Bigl(h^{k} \Bigl(\norm{ {u}}_{[H^{k+1}(\Omega)]^d} + \norm{p}_{H^{k}(\Omega)} \Bigr ) +\gamma^{1/2}_m |\delta {u}|_{\omega_M} \Bigr)(\norm{{w}}_{[H^{1}(\Omega)]^d}+\norm{r}_{L}).
\end{align}
As a consequence we can bound the quantity defined in (\ref{eq_13}) by
\begin{align}
    \langle r_f,&{w} \rangle_{V'_{0},V} +(r_g,r)_L \nonumber \\ &\leq C(h^{k} \left(\norm{ {u}}_{[H^{k+1}(\Omega)]^d} + \norm{p}_{H^{k}(\Omega)} \right ) +\gamma^{1/2}_m |\delta {u}|_{\omega_M})(\norm{{w}}_{[H^{1}(\Omega)]^d}+\norm{r}_{L}).
\end{align}
Since this bound holds for all $w \in V_0 \ \text{and} \  r \in L$, we conclude that
\begin{align}
    \norm{r_f}_{V'_{0}} +\norm{r_g}_{L}   &\leq C\Bigl(h^{k} \Bigl(\norm{ {u}}_{[H^{k+1}(\Omega)]^d} + \norm{p}_{H^{k}(\Omega)} \Bigr ) +\gamma^{1/2}_m |\delta {u}|_{\omega_M} \Bigr).
\end{align}
Using the Poincar\'e inequality (\ref{poincare_ineq}), we have the bound
\begin{align*}
  \norm{\zeta}_{L}  \leq C (|\zeta|_{\omega_M}+\norm{\nabla{\zeta}}_{L})  &\leq C h^{-k}\left(|h^{k}\zeta|_{\omega_M}+\norm{h^{k}\nabla{\zeta}}_{L}\right) \\ &\leq C h^{-k} \tnorm{(\zeta,0),(0,0)} \\ & \leq C h^{-k} \Bigl(h^{k} \Bigl(\norm{ {u}}_{[H^{k+1}(\Omega)]^d} + \norm{p}_{H^{k}(\Omega)} \Bigr ) +\gamma^{1/2}_m |\delta {u}|_{\omega_M}\Bigr) \\&\leq C \Bigl(\norm{ {u}}_{[H^{k+1}(\Omega)]^d} + \norm{p}_{H^{k}(\Omega)}  +\gamma^{1/2}_m h^{-k}|\delta {u}|_{\omega_M} \Bigr ).
\end{align*}
Thus, we can bound the terms in the right-hand side of (\ref{eq_131}) in the following way:
\begin{align}
    \norm{r_f}_{V'_{0}} +\norm{r_g}_{L} +\norm{\zeta}_L  &\leq C \Bigl(\norm{ {u}}_{[H^{k+1}(\Omega)]^d} + \norm{p}_{H^{k}(\Omega)}+\gamma^{1/2}_m h^{-k}|\delta {u}|_{\omega_M}  \Bigr ) .
\end{align}
And
\begin{align}
    \norm{r_f}_{V'_{0}} +\norm{r_g}_{L} +|\zeta|_{\omega_M}  &\leq Ch^{k} \Bigl(\norm{ {u}}_{[H^{k+1}(\Omega)]^d} + \norm{p}_{H^{k}(\Omega)} +\gamma^{1/2}_m |\delta {u}|_{\omega_M}  \Bigr ).
\end{align}
Using these two bounds in (\ref{eq_131}), we conclude that
\begin{align*}
    |\zeta|_{B}&\leq  ( \norm{ {u}}_{[H^{k+1}(\Omega)]^d} + \norm{p}_{H^{k}(\Omega)}  +\gamma^{1/2}_m h^{-k}|\delta {u}|_{\omega_M})^{1-\tau} (h^{k} \Bigl(\norm{ {u}}_{[H^{k+1}(\Omega)]^d} + \norm{p}_{H^{k}(\Omega)} \Bigr ) +\gamma^{1/2}_m |\delta {u}|_{\omega_M})^{\tau} \\ &\leq C h^{\tau k} (\norm{ {u}}_{[H^{k+1}(\Omega)]^d} + \norm{p}_{H^{k}(\Omega)}  +\gamma^{1/2}_m h^{-k}|\delta {u}|_{\omega_M}).
\end{align*}
which completes the proof of the theorem.\\

\section{Numerical simulations} \label{sec5}
In this section, we use several two-dimensional numerical examples to apply the methodology described in section \ref{fem_method}.  All experiments have been implemented using the open-source computing
platform FEniCSx \cite{BasixJoss22,UFL14}. A docker image to reproduce the numerical results is available at \url{https://doi.org/10.5281/zenodo.7442458}. The free parameters in (\ref{eq_03})--(\ref{eq_04}) are set to
\[\alpha=\gamma_u=\gamma_{div}=\gamma_{GLS}=\gamma^*_u=\gamma^*_p=10^{-1},\ \gamma_M=1000.\]
in all the numerical examples. In the first example we will verify the convergence orders for different polynomial orders using equal order interpolation $k$ for all variables. Then we consider the same test case using the minimal polynomial order that results in the same error bounds, $k_1=1$, $k_2 = \max\{k-1,1\}$ and $k_3=1$. Finally, we study the robustness of the error estimate with respect to the viscosity for a configuration where the target subdomain $B$ is strictly downwind the data subdomain $\omega_M$, so that every point $B$ is on a streamline intersecting $\omega_M$.
\subsection{Convergence study: Stokes example}\label{num_stoke}
To demonstrate the convergence behaviour of the method introduced in section \ref{fem_method}, we take the test case for the Stokes problem from \cite{Burman:2018:Hansbo}.
Let $\Omega = [0, 1]^2$ be the unit square. We
consider the velocity and pressure fields given by
\begin{align*}
	{\textbf{\rm \textbf{u}}}(x,y)&=(20xy^3,5x^4-5y^4)\\
	p(x,y)&=60x^2y-20y^3-5. 
\end{align*} 
It is simple to demonstrate that $(u, p)$ is a solution to the homogeneous Stokes problem with $\nu = 1$, corresponding to the system (\ref{stoke})--(\ref{stoke_1}) with $U = 0$ and $f = 0$. As a result, we consider (\ref{eq_03})--(\ref{eq_04}) with $U = 0$ and $f = 0$.  Two different geometric settings are considered: one in which the data is continued in the convex geometry, inside the convex hull of $\omega_M$, and one in which the solution is continued in the non-convex geometry, outside the convex hull of $\omega_M$. The convex geometry represented by Fig \ref{domaIn_2}(a) is given as:
$$ \omega_M = \Omega \setminus (0.1, 0.9) \times (0.25, 1), \ \ B = \Omega \setminus (0.1, 0.9)  \times (0.95, 1), $$
and the non-convex geometry represented by fig \ref{domaIn_2}(b) is given as:
$$ \omega_M = \{ (x,y) : 0.25 \leq x \leq 0.75, \  0.05 \leq y \leq 0.5\}, $$
$$ B = \{ (x,y) : 0.125 \leq x \leq 0.875, \  0.05 \leq y \leq 0.95\}. $$
\begin{figure}[!ht]
	\centerline{
		\begin{tabular}{cc}
			\hspace{0cm}
			\resizebox*{5.5cm}{!}{\includegraphics{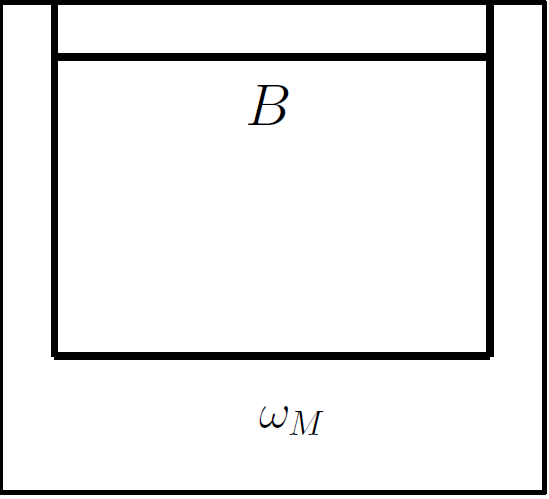}}%
			&\hspace{0.5cm}
			\resizebox*{5.5cm}{!}{\includegraphics{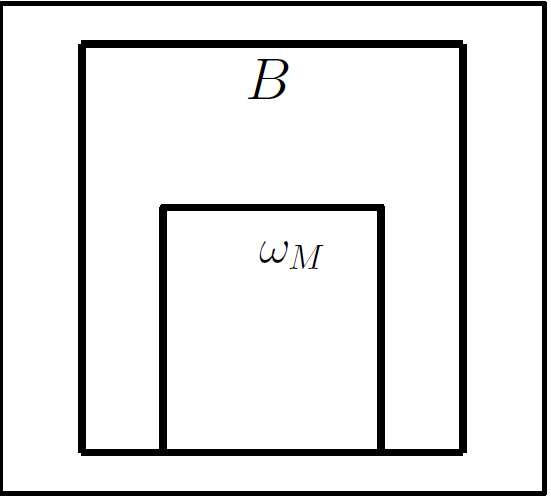}}%
			\\
			{(a)\ Convex geometry} \hspace{1cm}&{\hspace{0cm}(b) \ Non-convex geometry}
		\end{tabular}
	} \caption{{{Sketch of the domains used for computations in section \ref{num_stoke}.}}}.\label{domaIn_2}
\end{figure}

We begin by performing the computation using unperturbed data. The relative $L^2$- norm errors $\norm{u-u_h}_{[L^{2}(B)]^d}/\norm{u}_{[L^{2}(B)]^d}$ are computed in the subdomain $B$.   In addition, we present the history of convergence of the residual quantity for velocity stabilization:
\[\left(\gamma_u \sum_{F \in \mathcal{F}^{int}_h} \int_{F}h_F \sjump{(\nabla{u}_h -\nabla I_h{u}_h )\cdot {n}}^2 \ds \right)^{\frac{1}{2}}.\]
Fig. \ref{domain_3} displays the velocity, pressure errors and residual quantity in the convex and non-convex geometry.
Filled squares, circles and triangles represent the velocity errors; dashed lines represent the pressure error, and the plain thin lines represent the residual.  The expected order of convergence is observed for the residual in Lemma \ref{lemma_01}.  The local velocity error behaves consistently with the convergence rates obtained in Theorem \ref{th_01}.
We can also see in Fig. \ref{domain_3} that the higher order polynomials are more satisfactory for ill-posed problems.
   Next, we proceed with the numerical verification of the above method with data perturbation. Consider the perturbed data  
 \[u_M=u|_{\omega_{M}} +\delta u,\]
 with random perturbations
 \[|\delta u|_{\omega_M} =\mathcal{O}(h^{k-\theta}),\]
 for some $\theta \in \mathbb{N}_0 $  available for implementing our method. According to Theorem \ref{th_01}, we have the estimate
 \begin{align} \label{theta_cong}
 	|u-u_h|_{B}\leq C h^{k \tau- \theta}(\norm{u}_{[H^{k+1}(\Omega)]^d}+\norm{p}_{H^{k}(\Omega)}+1),
 \end{align}
consequently, convergence requires the condition $k \tau- \theta >0.$
Figs \ref{fig_data_convex}--\ref{fig_data_nonconvex} present the convergence history of the velocity, pressure and residual quantities with the data perturbation in the convex and non-convex geometry, respectively. The effect of different values of $\theta$ for relative $L^{2}$-error are studied in Figs \ref{fig_data_convex}--\ref{fig_data_nonconvex}. The relative error for $\theta=0$ is displayed in Figs.  \ref{fig_data_convex}(a) and \ref{fig_data_nonconvex}(a). In both cases, the results are in agreement with the Theorem \ref{th_01}. As stated in (\ref{theta_cong}), the p = 1 polynomial approximation may diverge for $\theta=1$, which is confirmed by Fig. \ref{fig_data_convex}(b). Next, the method $p=2$ converges linearly, whereas $p=3$ still manages to converge, albeit at a slower rate. As shown in the Fig. \ref{fig_data_convex}(c), this result is consistent with the fact that for $\theta=2$, convergence is no longer observed for any $p \leq 3$.
 Similar convergence results are observed in the non-convex domain as shown in Fig. \ref{fig_data_nonconvex}. The results of Figs. \ref{fig_data_convex}--\ref{fig_data_nonconvex} indicate that for the convex geometry $\tau \approx 1$ and for the non-convex geometry $\tau \approx \tfrac23$. In Fig. \ref{domain_31}-\ref{fig_data_nonconvex_1} the same results are presented for the case where the minimal polynomial order is considered that is $k_1=1$, $k_2=\max\{ k-1,1 \}$, $k_3=1$. The results are very similar and we conclude that for these numerical examples there is no disadvantage in taking the smallest possible dual space.

\begin{figure}[!ht]
	\centerline{
		\begin{tabular}{cc} 
			\hspace{0cm}
			\resizebox*{6cm}{!}{\includegraphics{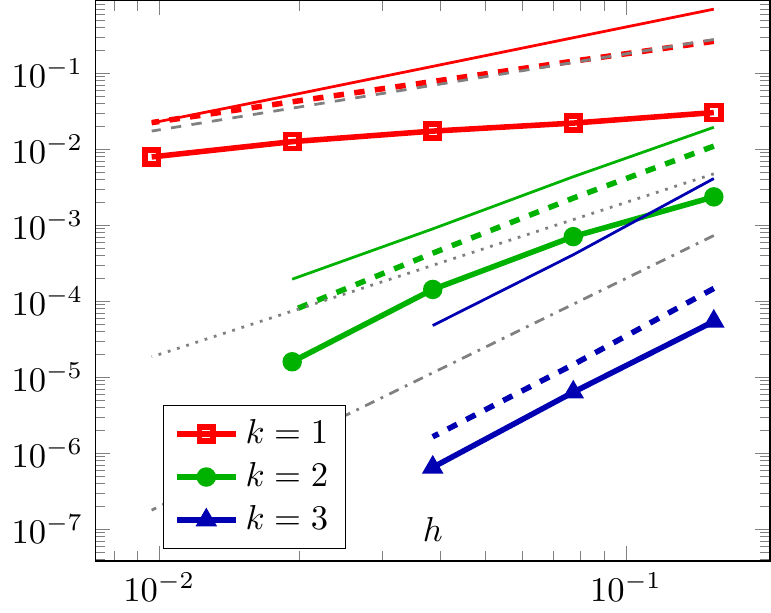}}%
			&\hspace{0cm}
			\resizebox*{6cm}{!}{\includegraphics{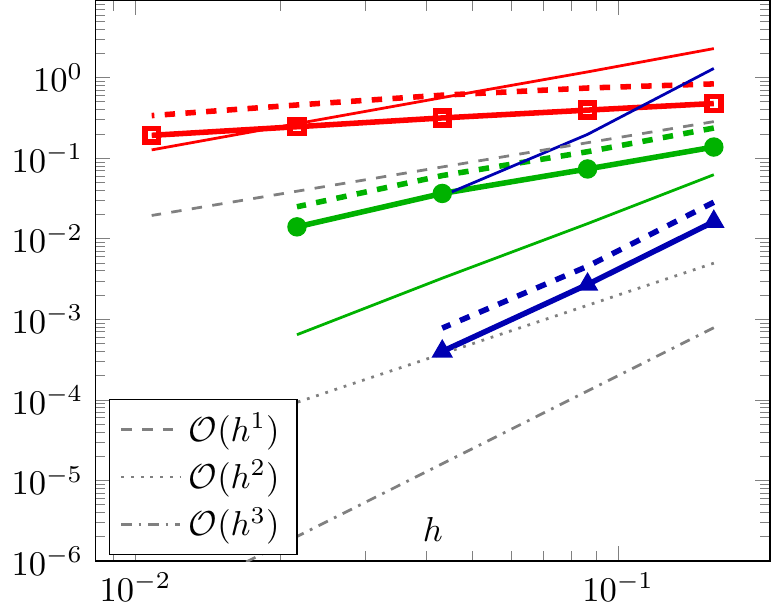}}%
			\\
			{(a)\ Errors for convex geometry in Fig. \ref{domaIn_2}(a)} \hspace{-1cm}&{(b) \ Errors for non-convex geometry in Fig. \ref{domaIn_2}(b)}
		\end{tabular}
	} \caption{{{Relative error for geometrical setup displayed in Fig. \ref{domaIn_2}. }}}\label{domain_3}
\end{figure}

\begin{figure}[ht] 
	\hspace{-0.7cm}
	\begin{subfigure}[b]{0.35\linewidth}
		\centering
		\includegraphics[width=0.75\linewidth]{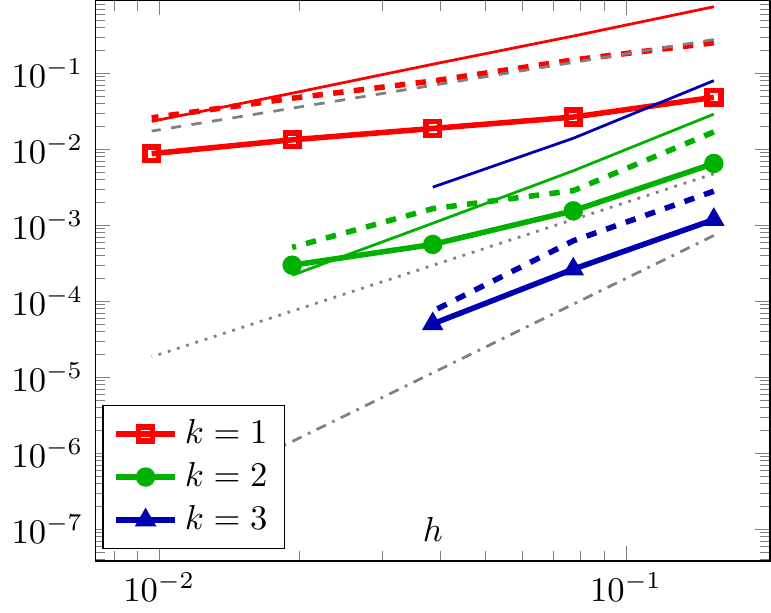} 
		\caption{$\theta=0$} 
	\end{subfigure}
	\hspace{-0.7cm}
	\begin{subfigure}[b]{0.35\linewidth}
		\centering
		\includegraphics[width=0.75\linewidth]{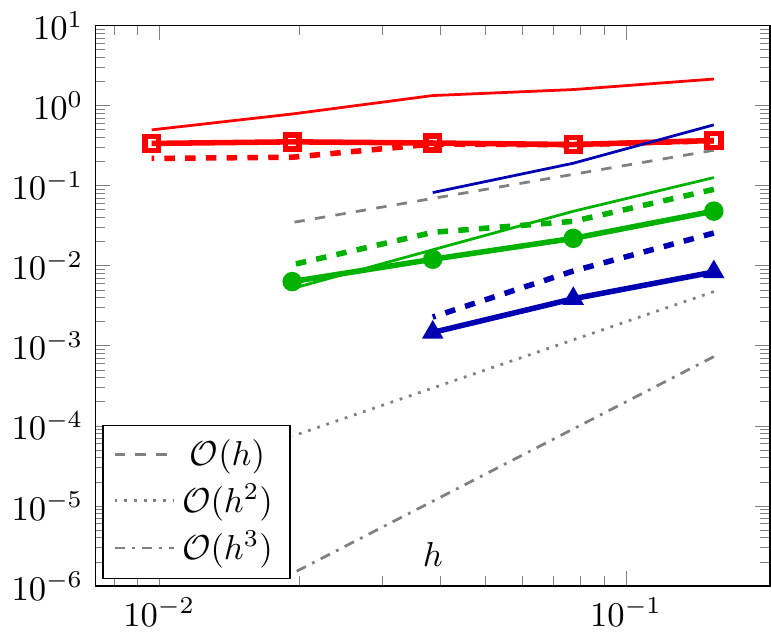} 
		\caption{$\theta=1$} 
	\end{subfigure} 
		\hspace{-0.7cm}
	\begin{subfigure}[b]{0.35\linewidth}
		\centering
		\includegraphics[width=0.75\linewidth]{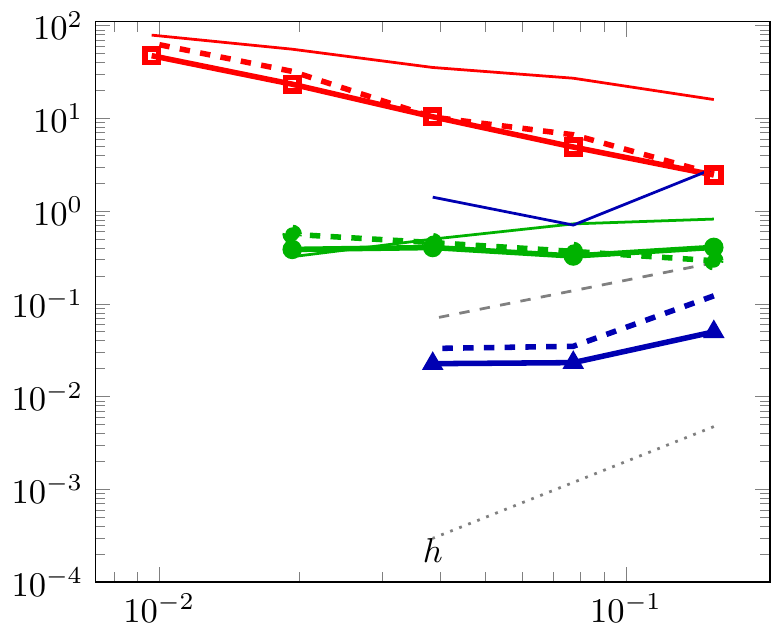} 
		\caption{$\theta=2$} 
	\end{subfigure}
	\caption{{ {Relative error for geometrical setup Fig. \ref{domaIn_2}(a) in terms of the strength of the data perturbation. }}}
	\label{fig_data_convex} 
\end{figure}

\begin{figure}[ht] 
	\hspace{-0.7cm}
	\begin{subfigure}[b]{0.35\linewidth}
		\centering
		\includegraphics[width=0.75\linewidth]{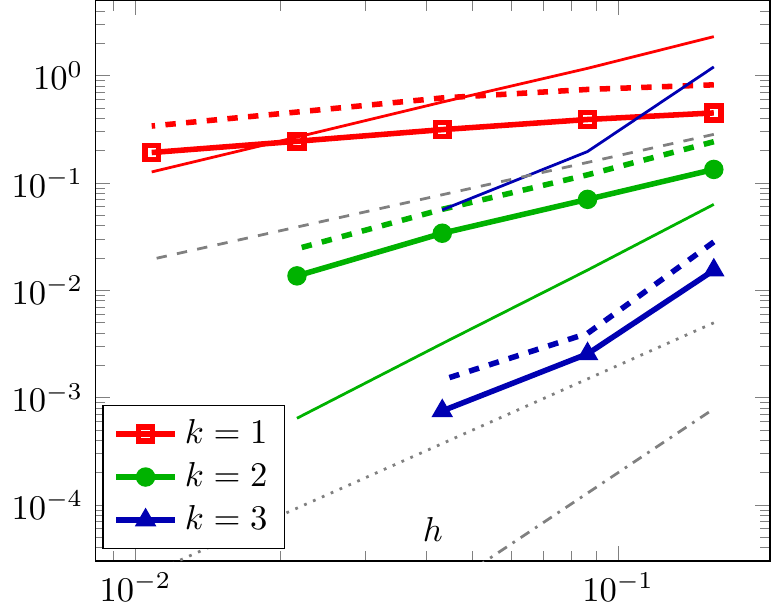} 
		\caption{$\theta=0$} 
	\end{subfigure}
	\hspace{-0.7cm}
	\begin{subfigure}[b]{0.35\linewidth}
		\centering
		\includegraphics[width=0.75\linewidth]{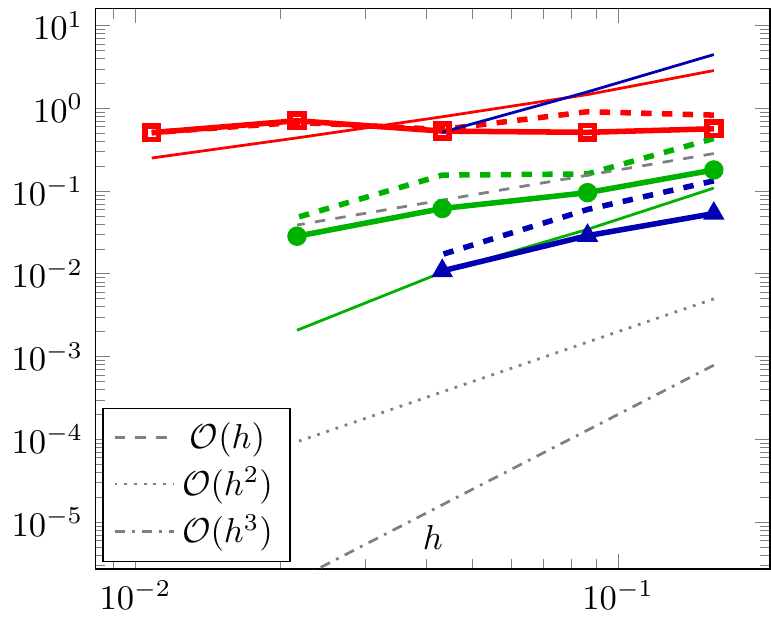} 
		\caption{$\theta=1$} 
	\end{subfigure} 
	\hspace{-0.7cm}
	\begin{subfigure}[b]{0.35\linewidth}
		\centering
		\includegraphics[width=0.75\linewidth]{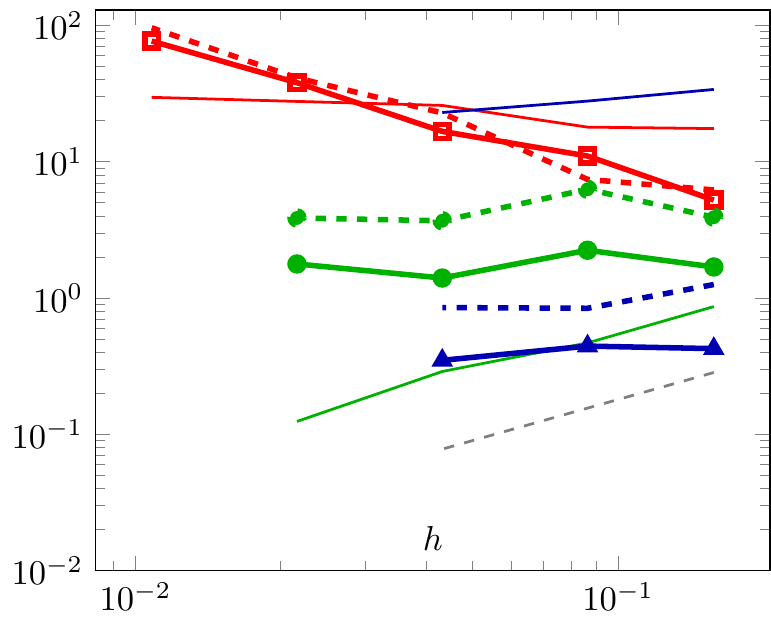} 
		\caption{$\theta=2$} 
	\end{subfigure}
	\caption{{ {Relative error for geometrical setup Fig. \ref{domaIn_2}(b) in terms of the strength of the data perturbation. }}}
	\label{fig_data_nonconvex} 
\end{figure}

\begin{figure}[!ht]
	\centerline{
		\begin{tabular}{cc} 
			\hspace{0cm}
			\resizebox*{6cm}{!}{\includegraphics{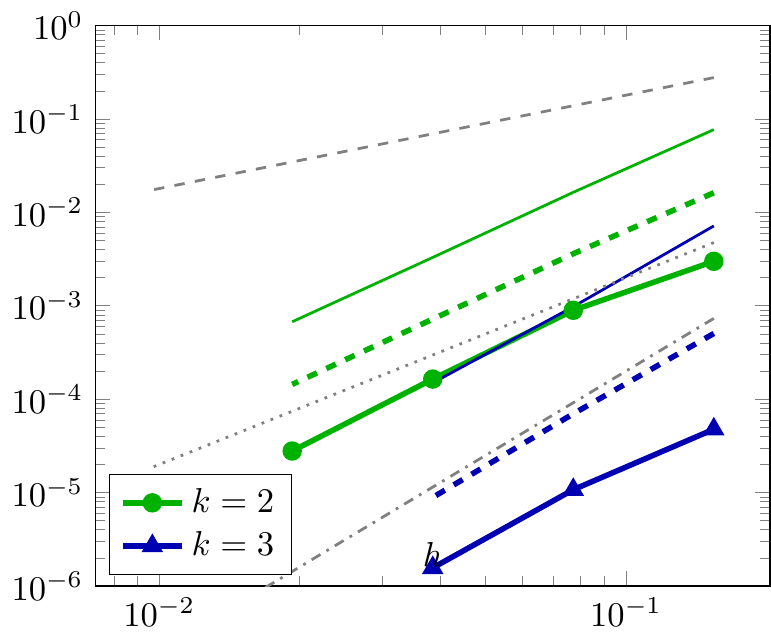}}%
			&\hspace{0cm}
			\resizebox*{6cm}{!}{\includegraphics{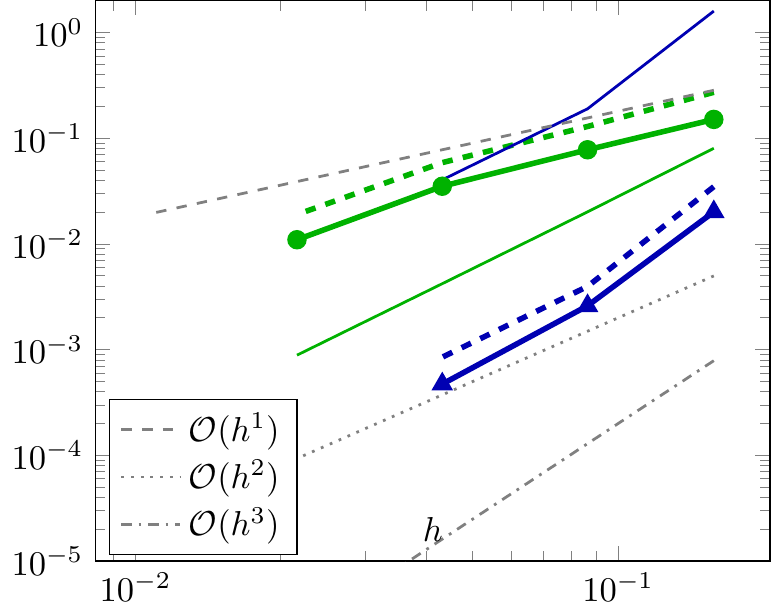}}%
			\\
			{(a)\ Errors for convex geometry in Fig. \ref{domaIn_2}(a)} \hspace{-1cm}&{(b) \ Errors for non-convex geometry in Fig. \ref{domaIn_2}(b)}
		\end{tabular}
	} \caption{{{Relative error with using the minimal polynomial order for geometrical setup displayed in Fig. \ref{domaIn_2}. }}}\label{domain_31}
\end{figure}

\begin{figure}[ht] 
	\hspace{-0.7cm}
	\begin{subfigure}[b]{0.35\linewidth}
		\centering
		\includegraphics[width=0.75\linewidth]{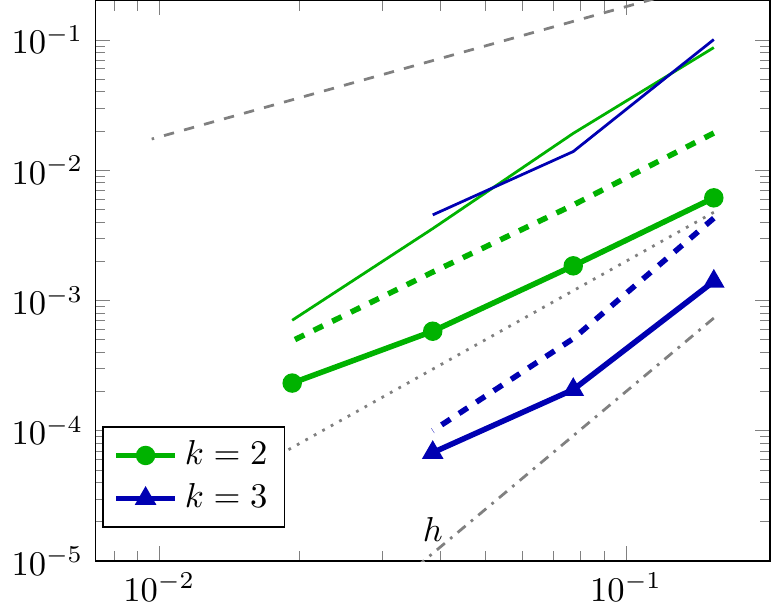} 
		\caption{$\theta=0$} 
	\end{subfigure}
	\hspace{-0.7cm}
	\begin{subfigure}[b]{0.35\linewidth}
		\centering
		\includegraphics[width=0.75\linewidth]{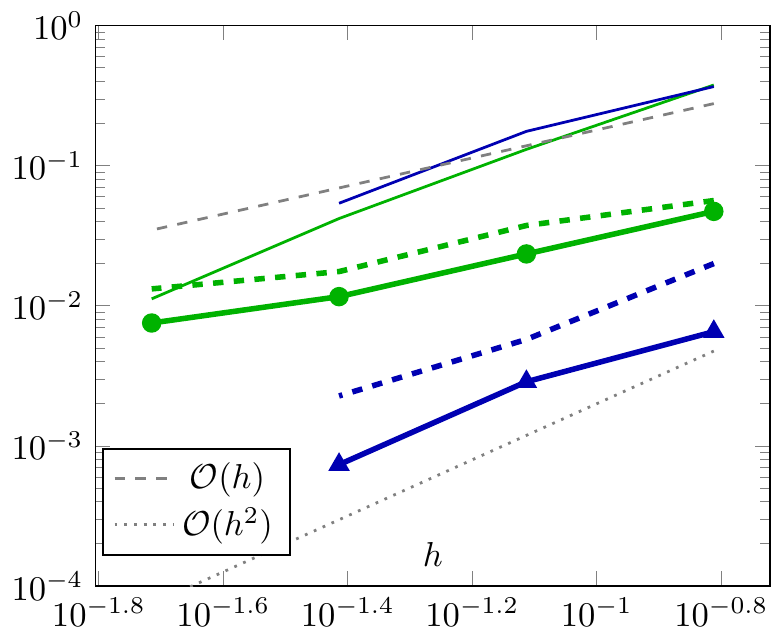} 
		\caption{$\theta=1$} 
	\end{subfigure} 
		\hspace{-0.7cm}
	\begin{subfigure}[b]{0.35\linewidth}
		\centering
		\includegraphics[width=0.75\linewidth]{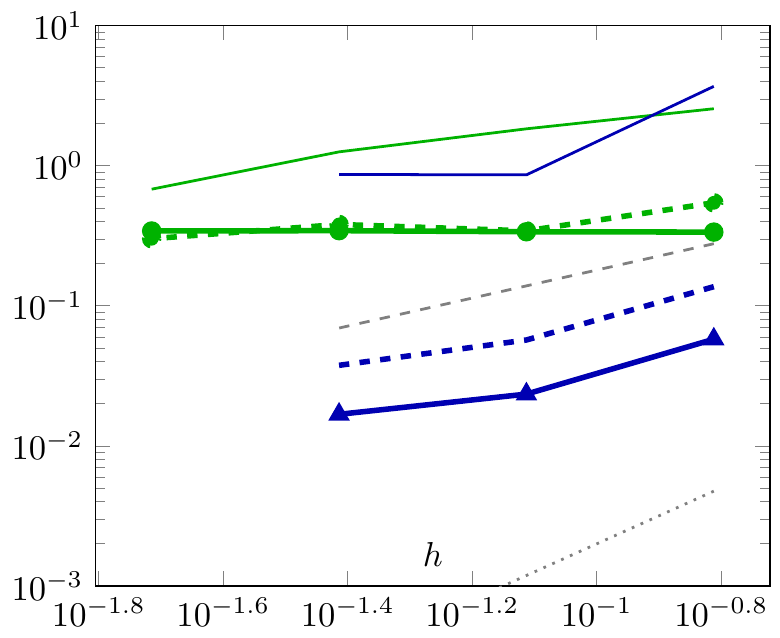} 
		\caption{$\theta=2$} 
	\end{subfigure}
	\caption{{ {Relative error with using the minimal polynomial order for geometrical setup Fig. \ref{domaIn_2}(a) in terms of the strength of the data perturbation. }}}
	\label{fig_data_convex_1} 
\end{figure}

\begin{figure}[ht] 
	\hspace{-0.7cm}
	\begin{subfigure}[b]{0.35\linewidth}
		\centering
		\includegraphics[width=0.75\linewidth]{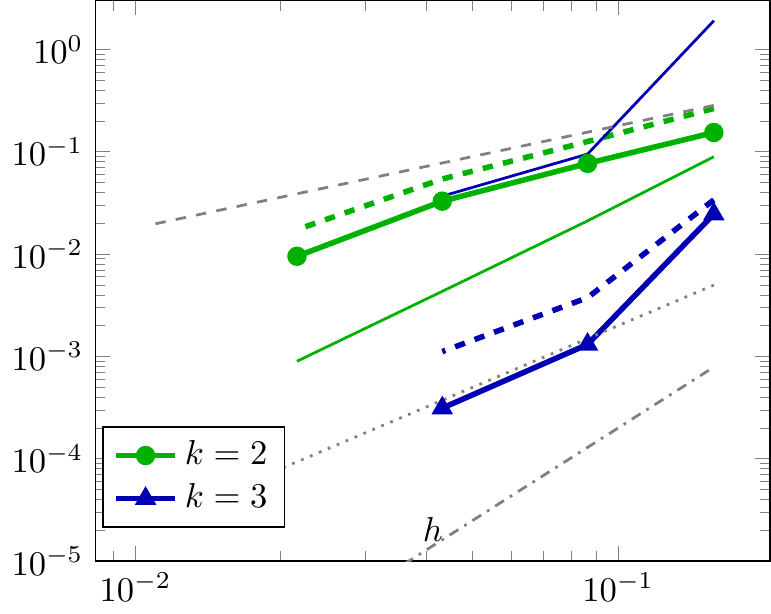} 
		\caption{$\theta=0$} 
	\end{subfigure}
	\hspace{-0.7cm}
	\begin{subfigure}[b]{0.35\linewidth}
		\centering
		\includegraphics[width=0.75\linewidth]{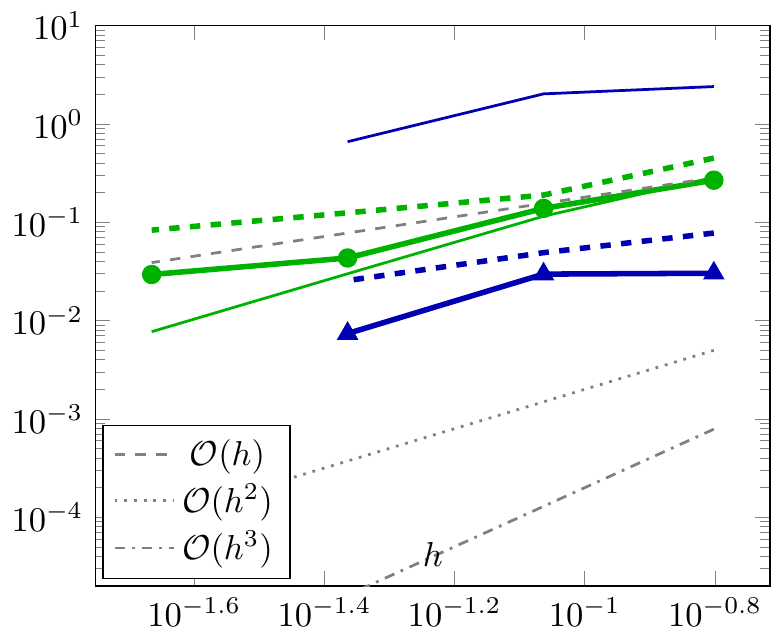} 
		\caption{$\theta=1$} 
	\end{subfigure} 
	\hspace{-0.7cm}
	\begin{subfigure}[b]{0.35\linewidth}
		\centering
		\includegraphics[width=0.75\linewidth]{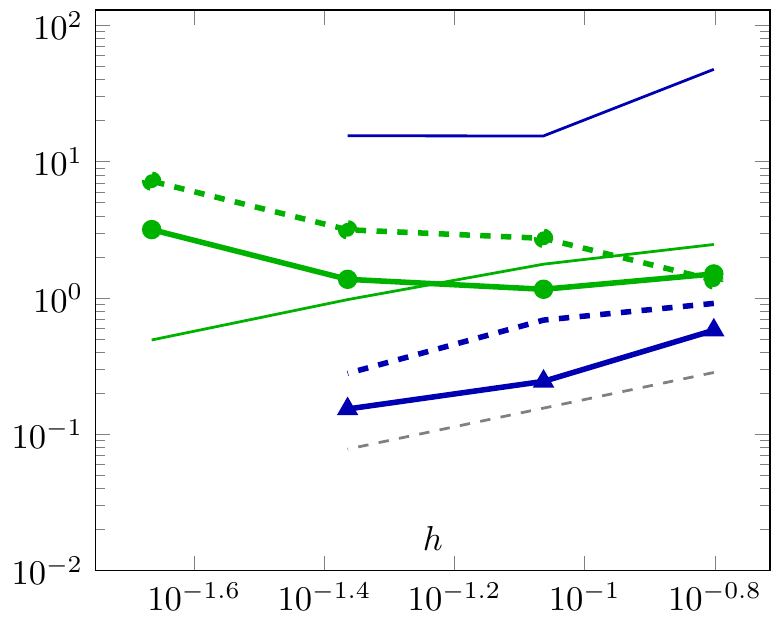} 
		\caption{$\theta=2$} 
	\end{subfigure}
	\caption{{ {Relative error with using the minimal polynomial order for geometrical setup Fig. \ref{domaIn_2}(b) in terms of the strength of the data perturbation. }}}
	\label{fig_data_nonconvex_1} 
\end{figure}
\subsection{Convergence study with varying viscosity}
In this subsection, we consider the flow of a viscous Newtonian fluid between two solid boundaries at $y = H,-H$ driven by a constant pressure gradient. The source term $f$ is chosen such that the solution of the plane Poiseuille flow 
\begin{align*}
	{{ {u}}}(x,y)=U(x,y)&=\left(\frac{P}{2 \mu}(H^2-y^2),0 \right),\\
	p(x,y)&=\left(\frac{1}{2}-x \right)P,
\end{align*} 
satisfies the model problem.
We demonstrate the performance of the numerical method for varying viscosity in a domain where the target subdomain is aligned with the flow, shown in Figure \ref{fig_new}, and defined by
\begin{align}\label{downstram}
  \omega_M = (0.0, 0.2) \times (0.2, 0.8), \ \ B =  (0.2, 0.8)  \times (0.45, 0.55).  
\end{align}
As in the previous section, we have examined the convergence of the method by performing numerical tests on both unperturbed and perturbed data. We vary the viscosity between $\nu=1$ and $\nu=0$. Observe that since no boundary conditions are imposed nothing needs to be changed in the formulation in the singular limit. Also note that the choice of $\xi_T$ and $\xi_F$ in \eqref{eq:Sg}-\eqref{eq:Sh} mimicks the choice for the stabilized method for (the well-posed) Oseen's problem used to improve robustness in the high Reynolds limit. Also with reference to high Reynolds computations for the well-posed case we here consider equal order interpolation for all fields. 

We wish to explore if the results on stability for the unique continuation for convection--diffusion equations in the limit of small diffusivity \cite{BNO20,BNO22} carry over to the case of incompressible flow. The key observation there was that for smooth solutions to the convection--diffusion equation the method had H\"older stable error estimates when diffusion dominates, similar to the analysis above, but in the convection dominated regime the stability in a subdomain slightly smaller than that spanned by the characteristics intersecting the data zone is Lipschitz. In that zone the convergence for the ill-posed problem coincides with that of the well-posed problem for piecewise affine approximation. 
As a means to study the effect of incompressibility we compare with the case where in addition to $u$ in $\omega_M$, $p$ is also provided as data in $\Omega$. The proposed method can be modified to accommodate this case by including $\frac{1}{2}\norm{p_h-p}_{\Omega}$,  as an additional term in the Lagrangian (\ref{eq4}). Note that when the pressure is added the velocity pressure coupling is strongly reduced. The relative $ L^2$-errors for running the same problem as above are displayed in Fig. \ref{fig_rel_error_1a}. Left side plots of Fig. \ref{fig_rel_error_1a} show the results without adding any additional pressure term, and right side plots of Fig. \ref{fig_rel_error_1a} display the results  by including the pressure data. We observe that  the results with pressure information are consistently better than those without. In particular for high order polynomials and high Reynolds number the information on the pressure appears to provide a very strong enhancement of the stability.
Further, the effect of data perturbations for different values of the viscosity coefficient is studied with and without the pressure augmentation, see Figs.  \ref{fig_data_downstream_1}--\ref{fig_data_downstream_1-1}.  We observe  that if  a priori information on the pressure is added and viscosity is reduced the convergence order for the relative $L^{2}$-error increases. This is consistent with the results of \cite{BNO20,BNO22}. If the pressure is not added however we do not observe this effect and it appears from these computational examples that we can not expect the result from \cite{BNO22} to hold for linearized incompressible flow.

In Fig. \ref{fig_data_downstream_1}-\ref{fig_data_downstream_1-1} we present the results under perturbations of data. These results show that the robustness under perturbations is also substantially enhanced if the pressure is known, indicating that the pressure velocity coupling introduces a strong sensitivity to perturbations.

\section{Conclusions}
We have introduced a finite element data assimilation method for the linearized Navier-Stokes' equation. We proved the natural extension of the error estimates  of \cite{Boulakia:2020:Burman} valid for piecewise affine approximation to the case of arbitrary polynomial orders. The expected increase in convergence rate was obtained, but the estimates also show that the sensitivity of the system to perturbations in data increase. The theoretical results were validated on some academic test cases. The main observations are that high order approximation for the ill-posed linearized Navier-Stokes' equations pays off, at least for sufficiently clean data. The spaces for the dual variables on the other hand can be chosen with piecewise affine approximation without loss of accuracy of the approximation. A study where the viscosity was varied showed that the incompressibility condition and the associated velocity-pressure coupling severely compromise the convective Lipschitz stability that is known to hold in the zone in the domain defined by points on the characteristics intersecting the data zone. If additional data in the form of global pressure measurements were added the results improved and were similar to the those of the scalar convection--diffusion equation. 

Future work will focus on the nonlinear case and the possibility of enhancing stability by adding knowledge of some other variable than the pressure, such as for example a passive tracer as in scalar image velocimetry \cite{BGO20}.

\begin{figure}[!ht]
	\centerline{
		\begin{tabular}{cc}
			\hspace{0cm}
			\resizebox*{6cm}{!}{\includegraphics{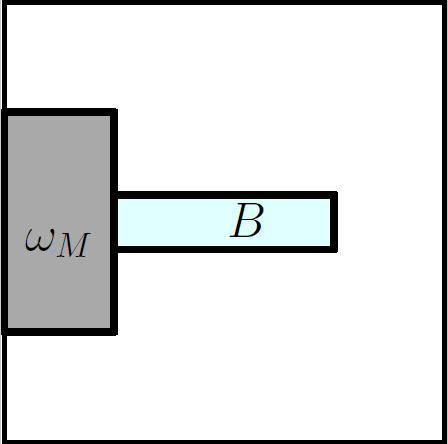}}%
		\end{tabular}
	} \caption{{{Data set $\omega_M$  and error measurement regions (B).}}} \label{fig_new}
\end{figure}

\begin{figure}[ht]

	\begin{subfigure}[b]{0.5\linewidth}
		\centering
		\includegraphics[width=0.75\linewidth]{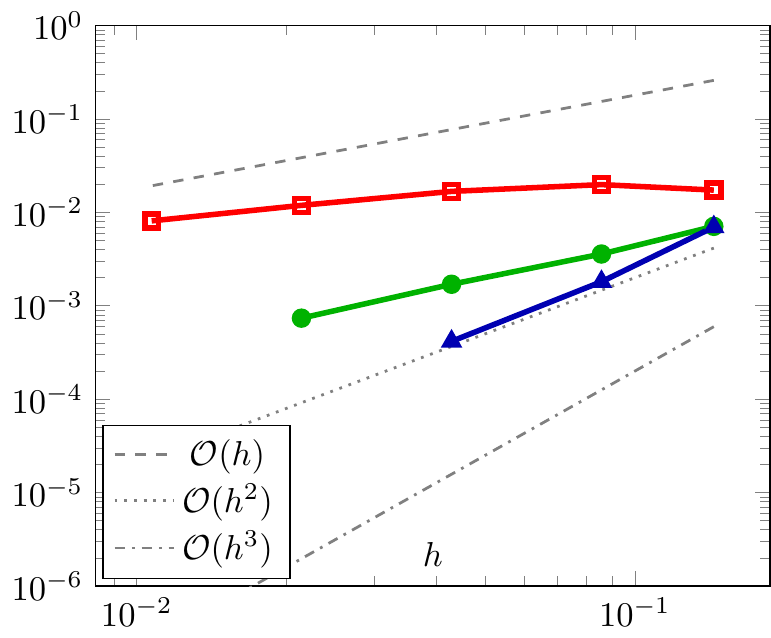} 
		\caption{ without pressure, $\nu=10^0$} 
	\end{subfigure}
	\begin{subfigure}[b]{0.5\linewidth}
		\centering
		\includegraphics[width=0.75\linewidth]{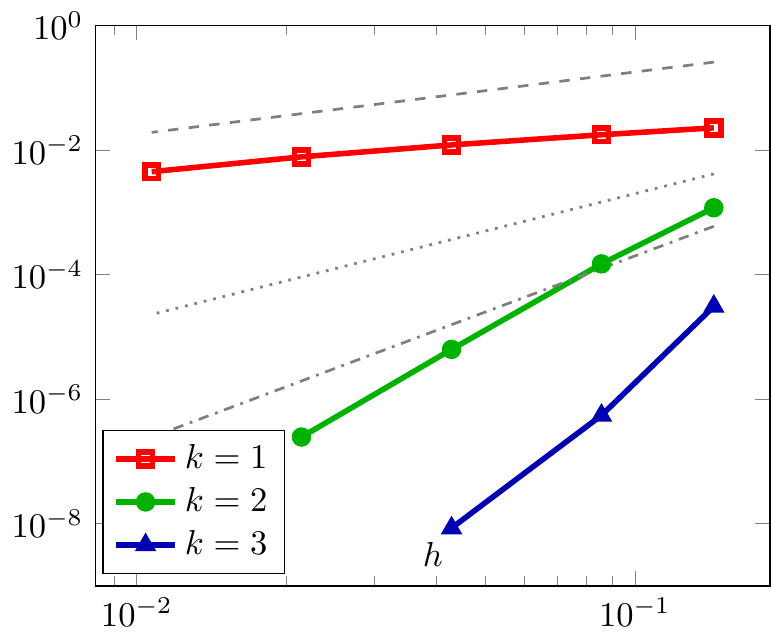} 
		\caption{ with pressure, $\nu=10^0$} 
	\end{subfigure}
	
	\begin{subfigure}[b]{0.5\linewidth}
		\centering
		\includegraphics[width=0.75\linewidth]{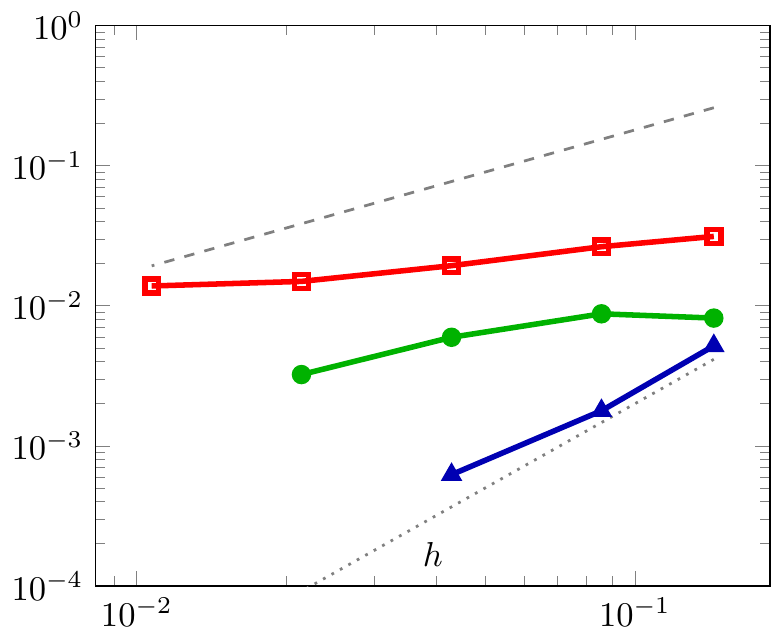} 
		\caption{ without pressure, $\nu=10^{-2}$} 
	\end{subfigure}
	\begin{subfigure}[b]{0.5\linewidth}
		\centering
		\includegraphics[width=0.75\linewidth]{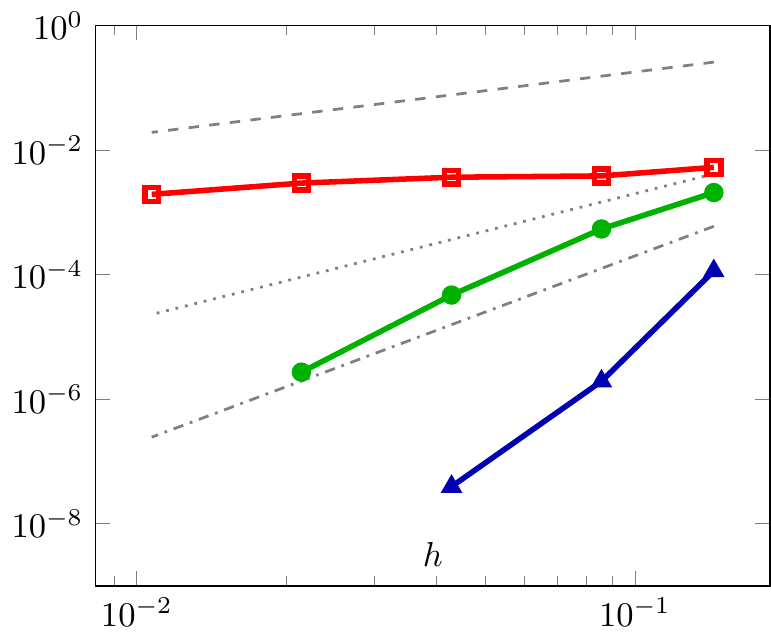} 
		\caption{ with pressure, $\nu=10^{-2}$} 
	\end{subfigure}

	\begin{subfigure}[b]{0.5\linewidth}
		\centering
		\includegraphics[width=0.75\linewidth]{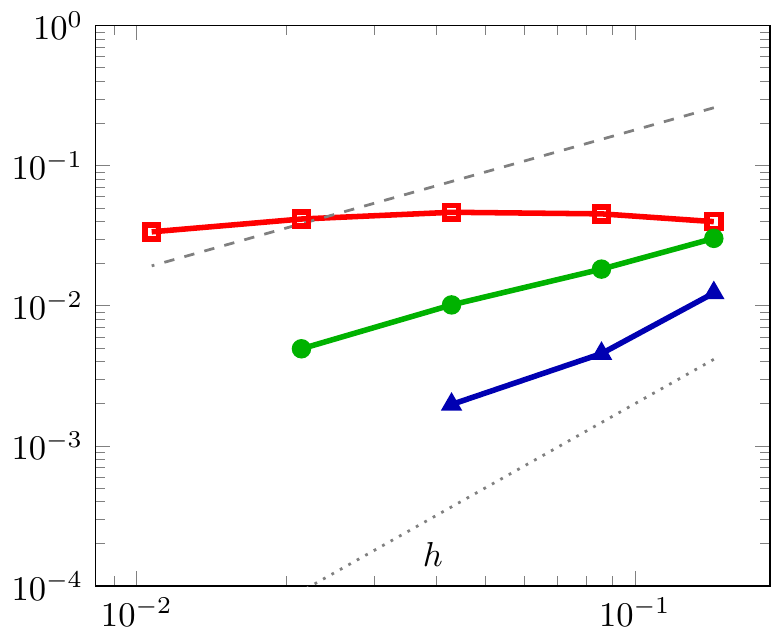} 
		\caption{ without pressure, $\nu=10^{-4}$} 
	\end{subfigure}
	\begin{subfigure}[b]{0.5\linewidth}
		\centering
		\includegraphics[width=0.75\linewidth]{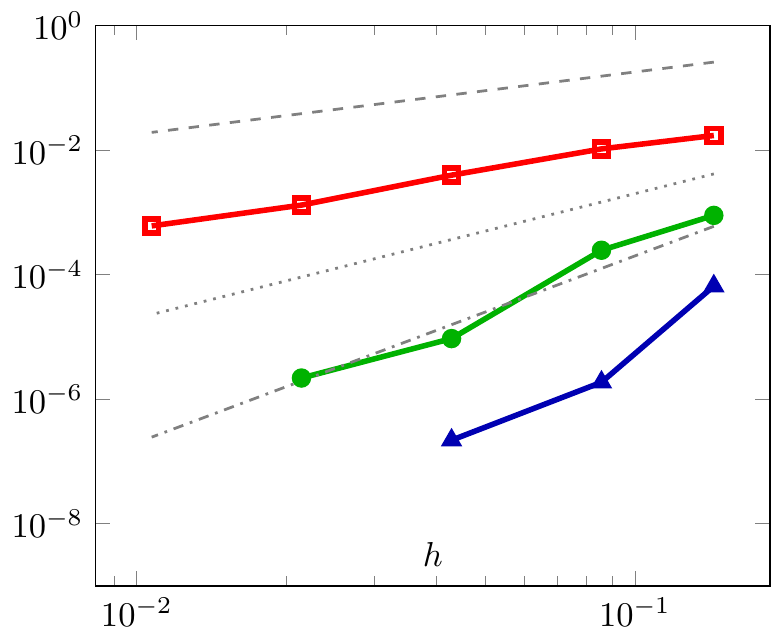} 
		\caption{ with pressure, $\nu=10^{-4}$} 
	\end{subfigure}
	
	\begin{subfigure}[b]{0.5\linewidth}
		\centering
		\includegraphics[width=0.75\linewidth]{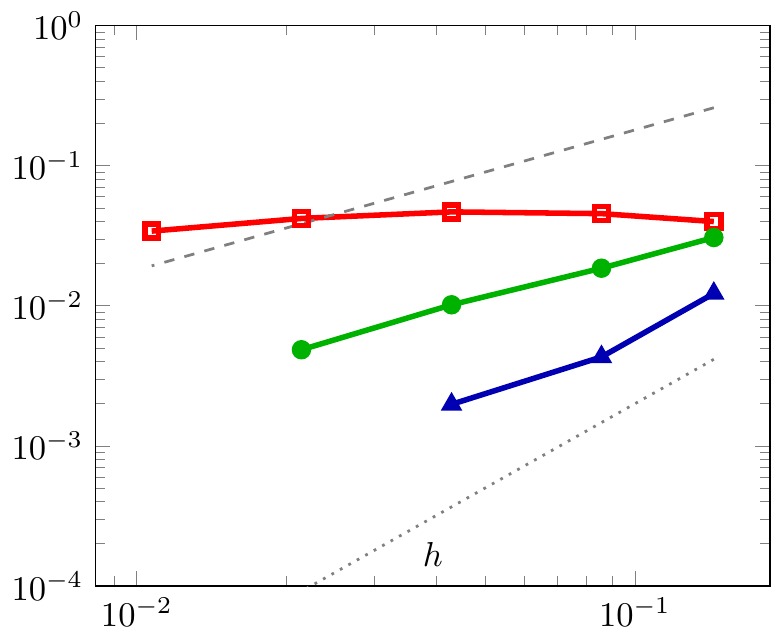} 
		\caption{ without pressure, $\nu=0$} 
	\end{subfigure}
	\begin{subfigure}[b]{0.5\linewidth}
		\centering
		\includegraphics[width=0.75\linewidth]{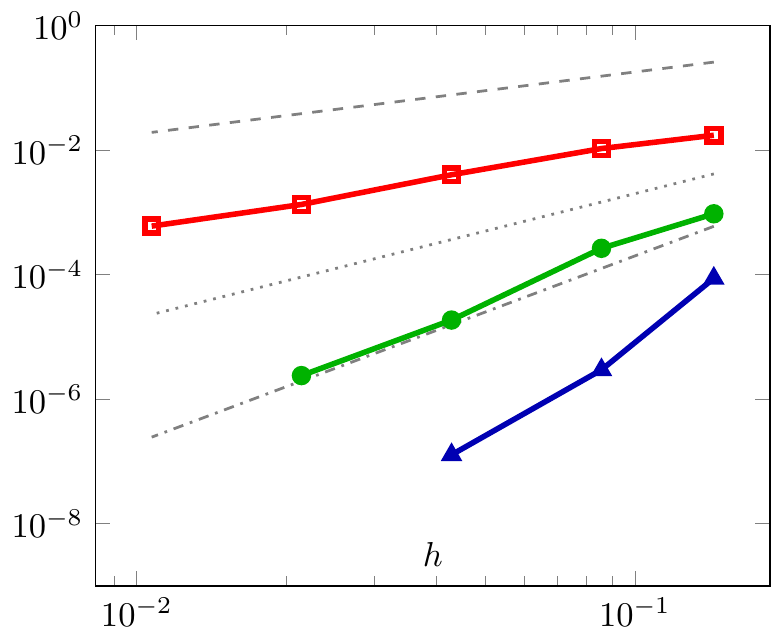} 
		\caption{ with pressure, $\nu=0$} 
	\end{subfigure}
	
	\caption{{ {Relative error for geometrical setup Fig. \ref{fig_new}. }}} \label{fig_rel_error_1a}
\end{figure}

	%
	%
	%
	%
	%
%

\begin{figure}[ht]

	\begin{subfigure}[b]{0.49\linewidth}
		\centering
		\includegraphics[width=0.75\linewidth]{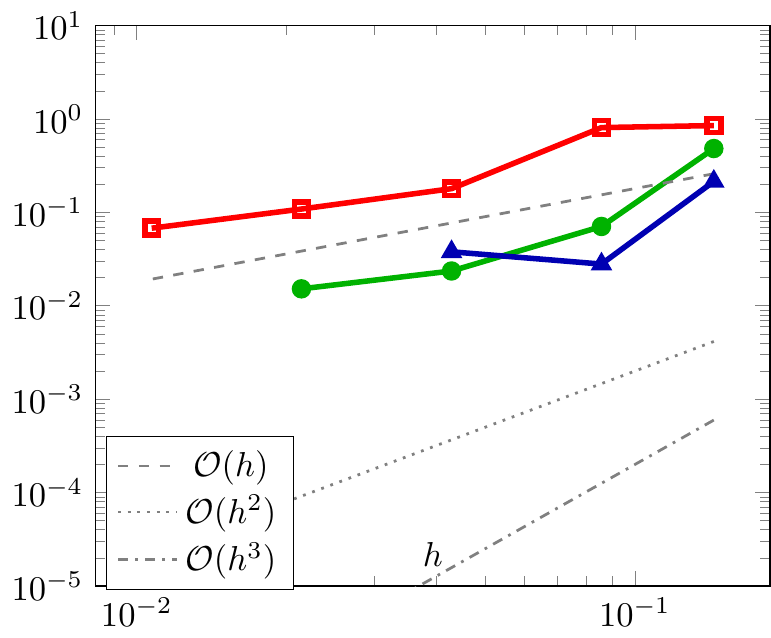} 
		\caption{without pressure $\theta=0$, $\nu=10^0$} 
	\end{subfigure}
 \begin{subfigure}[b]{0.49\linewidth}
		\centering
		\includegraphics[width=0.75\linewidth]{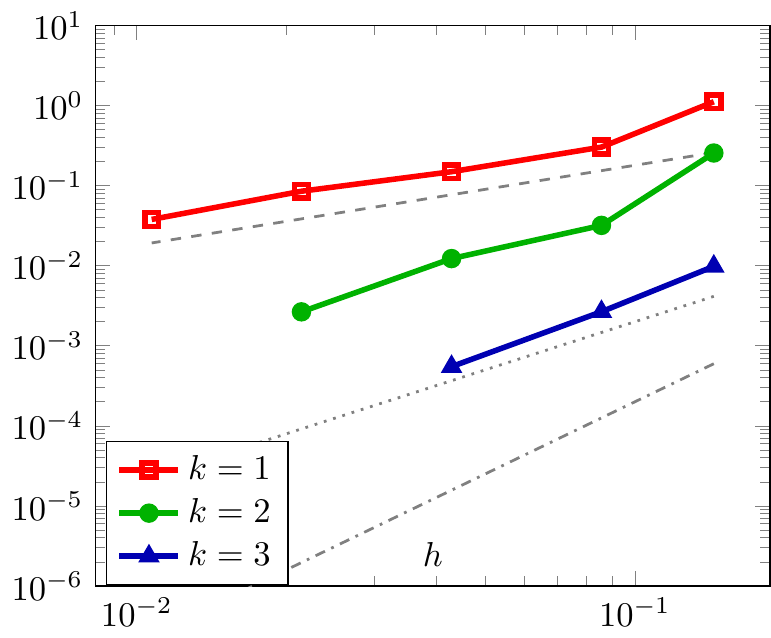} 
		\caption{with pressure $\theta=0$, $\nu=10^0$} 
	\end{subfigure}
	
 \begin{subfigure}[b]{0.49\linewidth}
		\centering
		\includegraphics[width=0.75\linewidth]{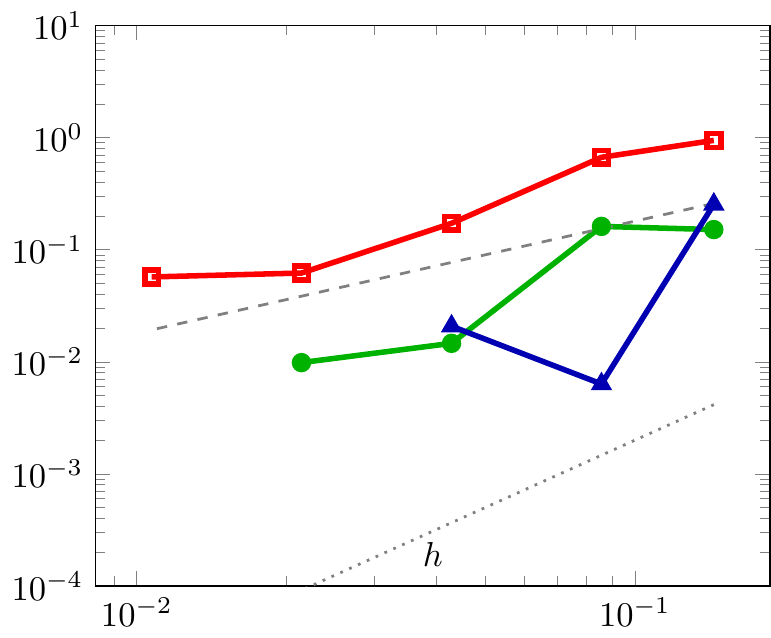} 
		\caption{without pressure  $\theta=0$, $\nu=10^{-2}$} 
	\end{subfigure} 
 \begin{subfigure}[b]{0.49\linewidth}
		\centering
		\includegraphics[width=0.75\linewidth]{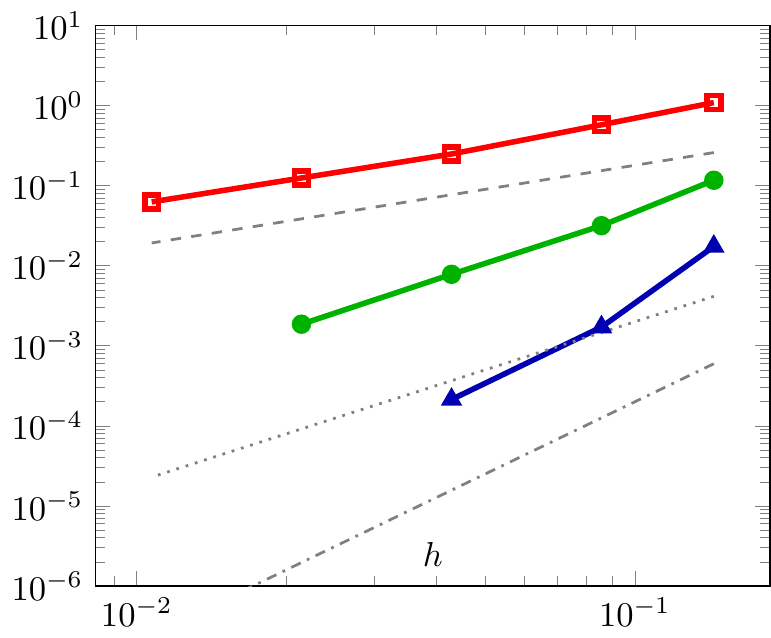} 
		\caption{with pressure $\theta=0$, $\nu=10^{-2}$} 
	\end{subfigure}
	
 \begin{subfigure}[b]{0.49\linewidth}
		\centering
		\includegraphics[width=0.75\linewidth]{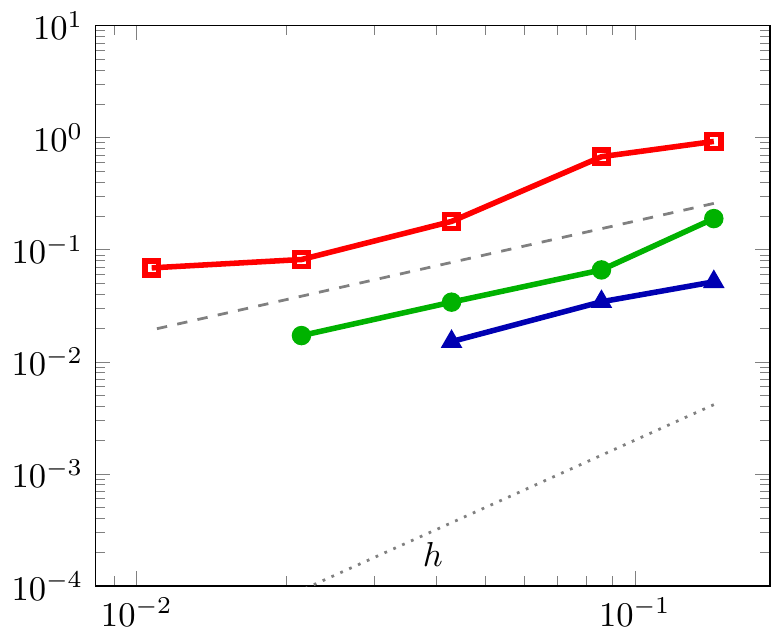} 
		\caption{without pressure  $\theta=0$, $\nu=10^{-4}$} 
	\end{subfigure}
	\begin{subfigure}[b]{0.49\linewidth}
		\centering
		\includegraphics[width=0.75\linewidth]{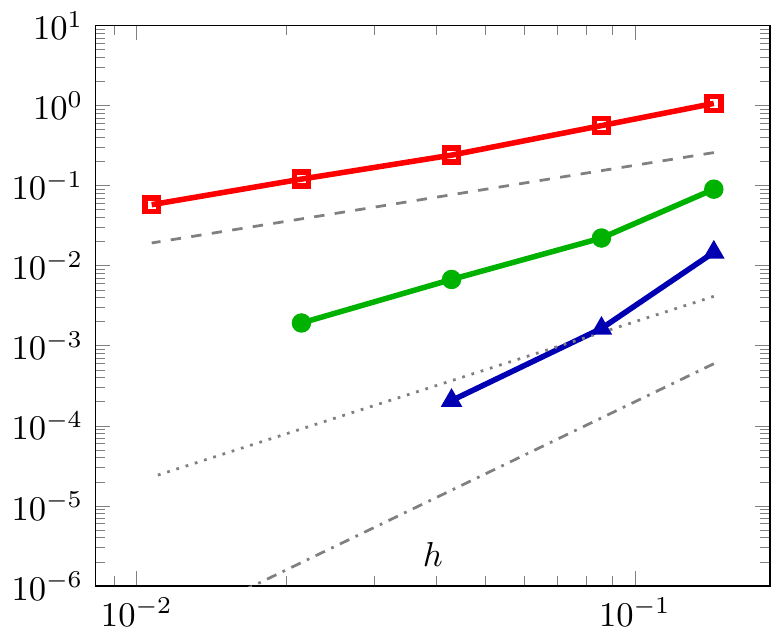} 
		\caption{with pressure $\theta=0$, $\nu=10^{-4}$} 
	\end{subfigure}

\begin{subfigure}[b]{0.49\linewidth}
		\centering
		\includegraphics[width=0.75\linewidth]{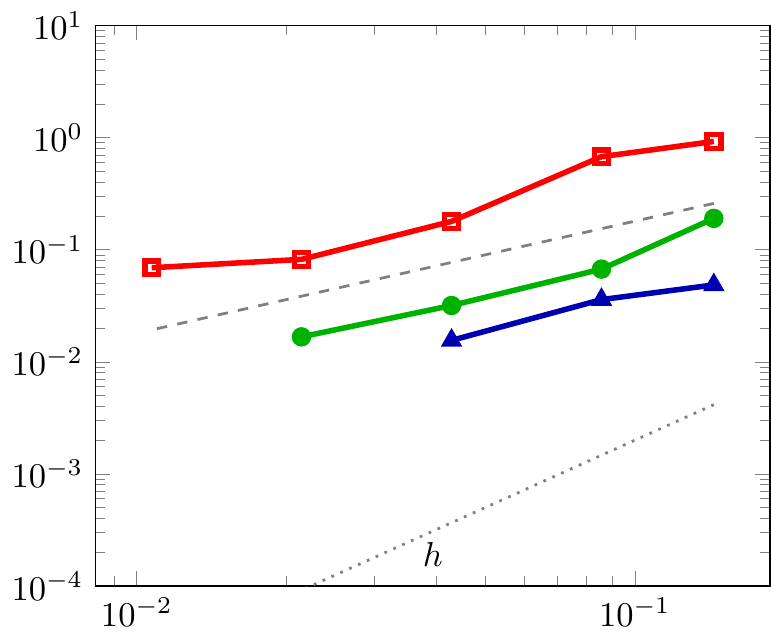} 
		\caption{without pressure $\theta=0$, $\nu=0$} 
	\end{subfigure}
\begin{subfigure}[b]{0.49\linewidth}
		\centering
		\includegraphics[width=0.75\linewidth]{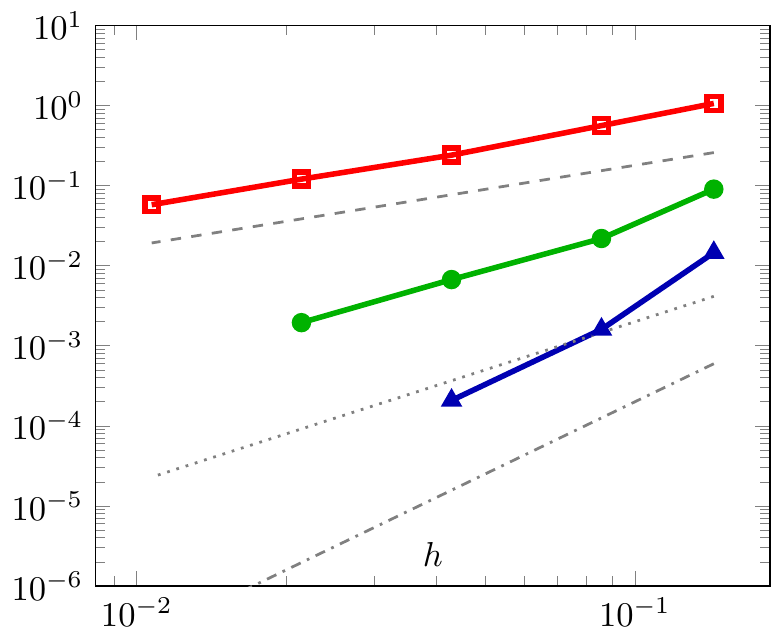} 
		\caption{with pressure $\theta=0$, $\nu=0$} 
	\end{subfigure}

	\caption{{ {Relative errors  in terms of the strength of the data perturbation for geometrical setup displayed in Fig. \ref{fig_new}. }}}
	\label{fig_data_downstream_1} 
\end{figure}

\begin{figure}[ht]

	\begin{subfigure}[b]{0.5\linewidth}
		\centering
		\includegraphics[width=0.75\linewidth]{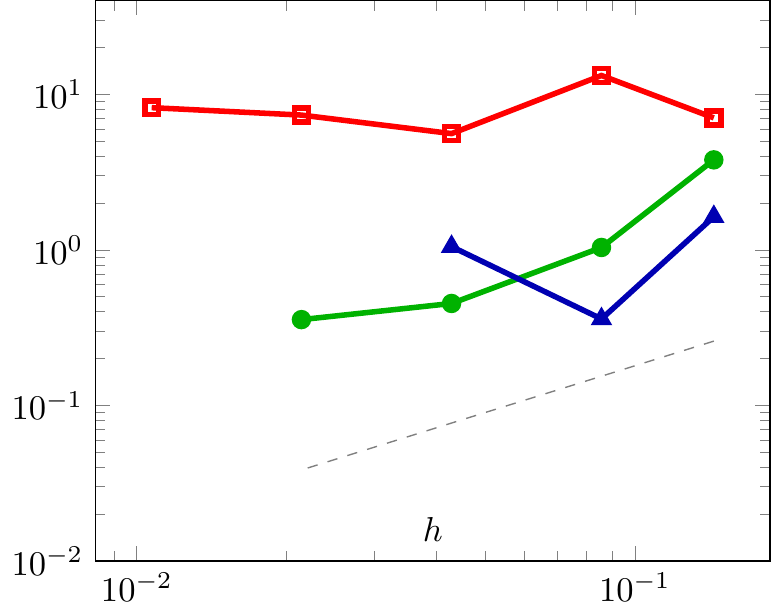} 
		\caption{without pressure $\theta=1$, $\nu=1$} 
	\end{subfigure}
 \begin{subfigure}[b]{0.5\linewidth}
		\centering
		\includegraphics[width=0.75\linewidth]{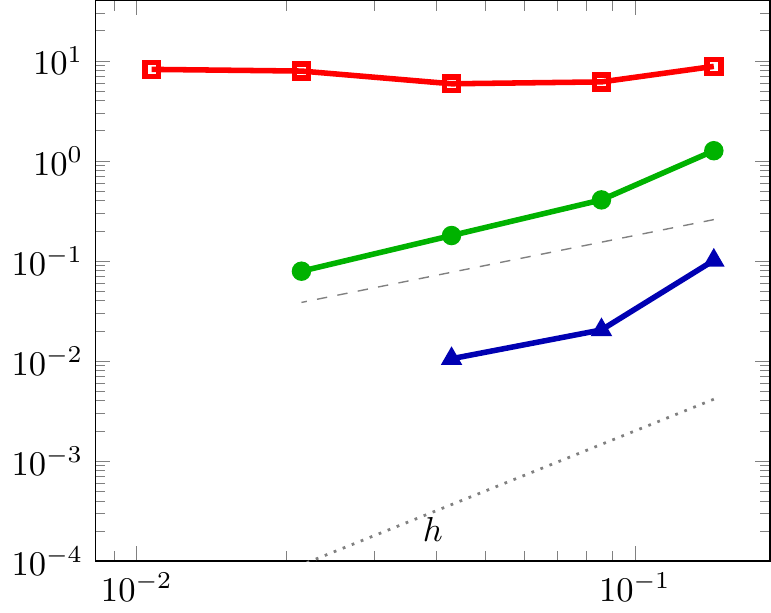} 
		\caption{with pressure $\theta=1$, $\nu=1$} 
	\end{subfigure}

 \begin{subfigure}[b]{0.5\linewidth}
		\centering
		\includegraphics[width=0.75\linewidth]{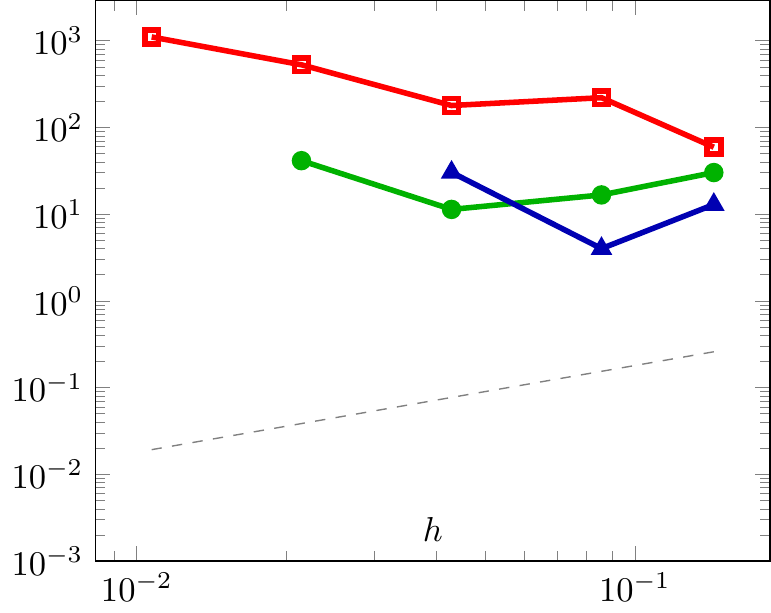} 
		\caption{without pressure $\theta=2$, $\nu=1$} 
	\end{subfigure}
 \begin{subfigure}[b]{0.5\linewidth}
		\centering
		\includegraphics[width=0.75\linewidth]{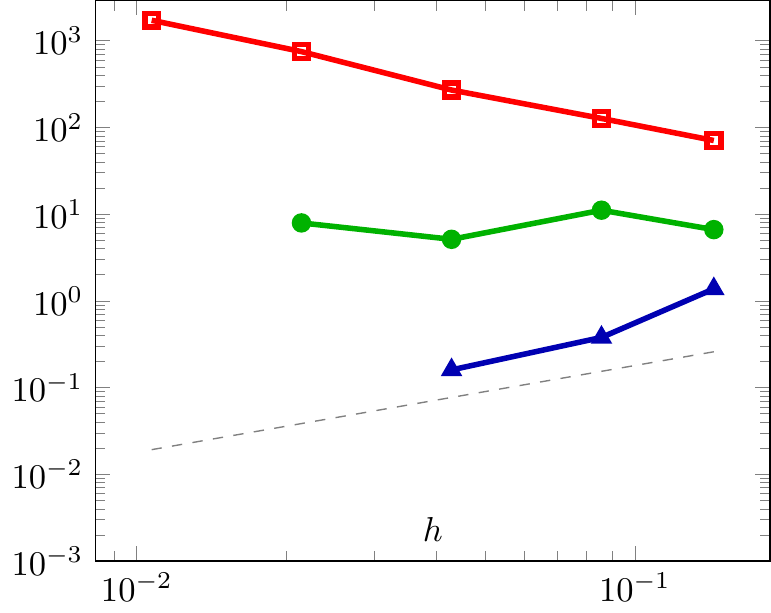} 
		\caption{with pressure $\theta=2$, $\nu=1$} 
	\end{subfigure}
	\caption{{ {Relative errors  in terms of the strength of the data perturbation  for geometrical setup displayed in  Fig. \ref{fig_new}. }}}
	\label{fig_data_downstream_1-1} 
\end{figure}

\section*{Acknowledgment}

This research was funded by EPSRC grants EP/T033126/1 and EP/V050400/1.


\bibliographystyle{plain}
\bibliography{bibtexexample}

\begin{thebibliography}{10}

\bibitem{ALE2009}
Giovanni Alessandrini, Luca Rondi, Edi Rosset, and Sergio Vessella.
\newblock The stability for the {C}auchy problem for elliptic equations.
\newblock {\em Inverse Problems}, 25(12):123004, 47, 2009.

\bibitem{UFL14}
Martin~S. Aln\ae{e}s, Anders Logg, Kristian~B. \O{}lgaard, Marie~E. Rognes, and
  Garth~N. Wells.
\newblock Unified {F}orm {L}anguage: A {D}omain-{S}ecific {L}anguage for {W}eak
  {F}ormulations of {P}artial {D}ifferential {E}quations.
\newblock {\em ACM Trans. Math. Softw.}, 40(2), mar 2014.

\bibitem{ASSS22}
Solveigh Averweg, Alexander Schwarz, Carina Schwarz, and J\"{o}rg Schr\"{o}der.
\newblock 3{D} modeling of generalized {N}ewtonian fluid flow with data
  assimilation using the least-squares finite element method.
\newblock {\em Comput. Methods Appl. Mech. Engrg.}, 392:Paper No. 114668, 19,
  2022.

\bibitem{BC2016}
Mehdi Badra, Fabien Caubet, and J\'{e}r\'{e}mi Dard\'{e}.
\newblock Stability estimates for {N}avier-{S}tokes equations and application
  to inverse problems.
\newblock {\em Discrete Contin. Dyn. Syst. Ser. B}, 21(8):2379--2407, 2016.

\bibitem{AB2010}
Andrea Ballerini.
\newblock Stable determination of an immersed body in a stationary {S}tokes
  fluid.
\newblock {\em Inverse Problems}, 26(12):125015, 25, 2010.

\bibitem{BIY2016}
Mourad Bellassoued, Oleg Imanuvilov, and Masahiro Yamamoto.
\newblock Carleman estimate for the {N}avier-{S}tokes equations and an
  application to a lateral {C}auchy problem.
\newblock {\em Inverse Problems}, 32(2):025001, 23, 2016.

\bibitem{BCFGM13}
C.~Bertoglio, D.~Chapelle, M.~A. Fern\'{a}ndez, J.-F. Gerbeau, and P.~Moireau.
\newblock State observers of a vascular fluid-structure interaction model
  through measurements in the solid.
\newblock {\em Comput. Methods Appl. Mech. Engrg.}, 256:149--168, 2013.

\bibitem{Boulakia:2020:Burman}
Muriel Boulakia, Erik Burman, Miguel~A. Fern\'{a}ndez, and Colette Voisembert.
\newblock Data assimilation finite element method for the linearized
  {N}avier-{S}tokes equations in the low {R}eynolds regime.
\newblock {\em Inverse Problems}, 36(8):085003, 21, 2020.

\bibitem{BEG2013}
Muriel Boulakia, Anne-Claire Egloffe, and C\'{e}line Grandmont.
\newblock Stability estimates for the unique continuation property of the
  {S}tokes system and for an inverse boundary coefficient problem.
\newblock {\em Inverse Problems}, 29(11):115001, 21, 2013.

\bibitem{BL2005}
L.~Bourgeois.
\newblock A mixed formulation of quasi-reversibility to solve the {C}auchy
  problem for {L}aplace's equation.
\newblock {\em Inverse Problems}, 21(3):1087--1104, 2005.

\bibitem{BL2006}
L.~Bourgeois.
\newblock Convergence rates for the quasi-reversibility method to solve the
  {C}auchy problem for {L}aplace's equation.
\newblock {\em Inverse Problems}, 22(2):413--430, 2006.

\bibitem{BD14}
Laurent Bourgeois and J\'{e}r\'{e}mi Dard\'{e}.
\newblock The ``exterior approach'' to solve the inverse obstacle problem for
  the {S}tokes system.
\newblock {\em Inverse Probl. Imaging}, 8(1):23--51, 2014.

\bibitem{BGO20}
E.~Burman, J.~J.~J. Gillissen, and L.~Oksanen.
\newblock Stability estimate for scalar image velocimetry, 2020.

\bibitem{EB2013}
Erik Burman.
\newblock Stabilized finite element methods for nonsymmetric, noncoercive, and
  ill-posed problems. {P}art {I}: {E}lliptic equations.
\newblock {\em SIAM J. Sci. Comput.}, 35(6):A2752--A2780, 2013.

\bibitem{EB2014}
Erik Burman.
\newblock Error estimates for stabilized finite element methods applied to
  ill-posed problems.
\newblock {\em C. R. Math. Acad. Sci. Paris}, 352(7-8):655--659, 2014.

\bibitem{EB2016}
Erik Burman.
\newblock Stabilised finite element methods for ill-posed problems with
  conditional stability.
\newblock In {\em Building bridges: connections and challenges in modern
  approaches to numerical partial differential equations}, volume 114 of {\em
  Lect. Notes Comput. Sci. Eng.}, pages 93--127. Springer, [Cham], 2016.

\bibitem{EB2017}
Erik Burman.
\newblock A stabilized nonconforming finite element method for the elliptic
  {C}auchy problem.
\newblock {\em Math. Comp.}, 86(303):75--96, 2017.

\bibitem{Burman:2018:Hansbo}
Erik Burman and Peter Hansbo.
\newblock Stabilized nonconforming finite element methods for data assimilation
  in incompressible flows.
\newblock {\em Math. Comp.}, 87(311):1029--1050, 2018.

\bibitem{BNO20}
Erik Burman, Mihai Nechita, and Lauri Oksanen.
\newblock A stabilized finite element method for inverse problems subject to
  the convection-diffusion equation. {I}: diffusion-dominated regime.
\newblock {\em Numer. Math.}, 144(3):451--477, 2020.

\bibitem{BNO22}
Erik Burman, Mihai Nechita, and Lauri Oksanen.
\newblock A stabilized finite element method for inverse problems subject to
  the convection-diffusion equation. {II}: convection-dominated regime.
\newblock {\em Numer. Math.}, 150(3):769--801, 2022.

\bibitem{DH2013}
J\'{e}r\'{e}mi Dard\'{e}, Antti Hannukainen, and Nuutti Hyv\"{o}nen.
\newblock An {$H_{{div}}$}-based mixed quasi-reversibility method for solving
  elliptic {C}auchy problems.
\newblock {\em SIAM J. Numer. Anal.}, 51(4):2123--2148, 2013.

\bibitem{DPV12}
Marta D'Elia, Mauro Perego, and Alessandro Veneziani.
\newblock A variational data assimilation procedure for the incompressible
  {N}avier-{S}tokes equations in hemodynamics.
\newblock {\em J. Sci. Comput.}, 52(2):340--359, 2012.

\bibitem{Ern:2012:DGBook}
Daniele~Antonio Di~Pietro and Alexandre Ern.
\newblock {\em Mathematical aspects of discontinuous Galerkin methods},
  volume~69.
\newblock Springer Science \& Business Media, 2011.

\bibitem{Ern:2004:FEM}
A~Ern and JL~Guermond.
\newblock Theory and practice of finite elements springer-verlag.
\newblock {\em New York}, 2004.

\bibitem{fabre1996prolongement}
Caroline Fabre and Gilles Lebeau.
\newblock Prolongement unique des solutions de l'equation de {S}tokes.
\newblock {\em Comm. Partial Differential Equations}, 21(3-4):573--596, 1996.

\bibitem{GN20}
Bosco Garc\'{\i}a-Archilla and Julia Novo.
\newblock Error analysis of fully discrete mixed finite element data
  assimilation schemes for the {N}avier-{S}tokes equations.
\newblock {\em Adv. Comput. Math.}, 46(4):Paper No. 61, 33, 2020.

\bibitem{HMM10}
J.~J. Heys, T.~A. Manteuffel, S.~F. McCormick, M.~Milano, J.~Westerdale, and
  M.~Belohlavek.
\newblock Weighted least-squares finite elements based on particle imaging
  velocimetry data.
\newblock {\em J. Comput. Phys.}, 229(1):107--118, 2010.

\bibitem{IM2015}
O.~Yu. Imanuvilov and M.~Yamamoto.
\newblock Global uniqueness in inverse boundary value problems for the
  {N}avier-{S}tokes equations and {L}am\'{e} system in two dimensions.
\newblock {\em Inverse Problems}, 31(3):035004, 46, 2015.

\bibitem{IYM2015}
O.~Yu. Imanuvilov and M.~Yamamoto.
\newblock Remark on boundary data for inverse boundary value problems for the
  {N}avier-{S}tokes equations [{A}ddendum to {MR}3319370].
\newblock {\em Inverse Problems}, 31(10):109401, 4, 2015.

\bibitem{IV2006}
Victor Isakov.
\newblock {\em Inverse problems for partial differential equations}, volume 127
  of {\em Applied Mathematical Sciences}.
\newblock Springer, New York, second edition, 2006.

\bibitem{IB2015}
Kazufumi Ito and Bangti Jin.
\newblock {\em Inverse problems}, volume~22 of {\em Series on Applied
  Mathematics}.
\newblock World Scientific Publishing Co. Pte. Ltd., Hackensack, NJ, 2015.
\newblock Tikhonov theory and algorithms.

\bibitem{john1960continuous}
Fritz John.
\newblock Continuous dependence on data for solutions of partial differential
  equations with a prescribed bound.
\newblock {\em Comm. Pure Appl. Math.}, 13(4):551--585, 1960.

\bibitem{LL1967}
R.~Latt\`es and J.-L. Lions.
\newblock {\em M\'{e}thode de quasi-r\'{e}versibilit\'{e} et applications}.
\newblock Travaux et Recherches Math\'{e}matiques, No. 15. Dunod, Paris, 1967.

\bibitem{LUW2010}
Ching-Lung Lin, Gunther Uhlmann, and Jenn-Nan Wang.
\newblock Optimal three-ball inequalities and quantitative uniqueness for the
  {S}tokes system.
\newblock {\em Discrete Contin. Dyn. Syst.}, 28(3):1273--1290, 2010.

\bibitem{SD18}
Alexander Schwarz and Richard~P. Dwight.
\newblock Data assimilation for {N}avier-{S}tokes using the least-squares
  finite-element method.
\newblock {\em Int. J. Uncertain. Quantif.}, 8(5):383--403, 2018.

\bibitem{BasixJoss22}
Matthew~W. Scroggs, Igor~A. Baratta, Chris~N. Richardson, and Garth~N. Wells.
\newblock Basix: a runtime finite element basis evaluation library.
\newblock {\em J. Open Source Softw.}, 7(73):3982, 2022.

\bibitem{TA1977}
Andrey~N. Tikhonov and Vasiliy~Y. Arsenin.
\newblock {\em Solutions of ill-posed problems}.
\newblock Scripta Series in Mathematics. V. H. Winston \& Sons, Washington,
  D.C.; John Wiley \& Sons, New York-Toronto, Ont.-London, 1977.
\newblock Translated from the Russian, Preface by translation editor Fritz
  John.

\end{thebibliography}

\end{document}